\documentclass[reqno,11pt]{amsart}
\usepackage[margin=1.1in, footskip=1cm]{geometry}
\usepackage{amsmath, amsthm, amssymb, amsfonts, booktabs}
\usepackage{mathrsfs}
\usepackage{thmtools, thm-restate}
\usepackage{xcolor}
\allowdisplaybreaks
\pdfpagewidth=\paperwidth
\pdfpageheight=\paperheight
\usepackage[initials,nobysame]{amsrefs}

\newtheorem{theorem}{Theorem}[section]
\newtheorem{lemma}[theorem]{Lemma}
\newtheorem{corollary}[theorem]{Corollary}

\theoremstyle{definition}

\theoremstyle{remark}
\newtheorem{remark}[theorem]{Remark}

\numberwithin{equation}{section}

\newcommand{\mmod}[1]{\,\,({\rm{mod}}\,\,#1)}

\def\calX{{\mathcal X}}
\def\calY{{\mathcal Y}}

\def\N{{\mathbb N}}
\def\R{{\mathbb R}}\def\T{{\mathbb T}}
\def\Z{{\mathbb Z}}

\def\grM{{\mathfrak M}}

\newcommand{\hA}{\widehat{A}}

\newcommand{\hC}{\widehat{C}}
\newcommand{\hD}{\widehat{D}}

\newcommand{\hG}{\widehat{G}}
\newcommand{\hH}{\widehat{H}}

\newcommand{\hi}{\hat{i}}

\newcommand{\hL}{\widehat{L}}

\newcommand{\hY}{\widehat{Y}}

\renewcommand{\d}{\,{\rm d}} 

\renewcommand{\epsilon}{\varepsilon}

\newcommand{\<}{\begin{equation}}
\renewcommand{\>}{\end{equation}}

\makeatletter
\@namedef{subjclassname@2020}{\textup{2020} Mathematics Subject Classification}
\makeatother

\begin{document}

\title{Subconvexity of Short $k$-Free Exponential Sums}
\author[Ben Doyle]{Ben Doyle}
\address{Department of Mathematics, Purdue University, 150 N. University Street, West 
Lafayette, IN 47907-2067, USA}
\email{doyle133@purdue.edu}
\thanks{}
\date{}

\begin{abstract}
    Let $S_k(\alpha;K)$ denote the exponential sum over $k$-free integers in the short interval $(N-K,N]$. For $s>0$, we prove essentially tight bounds on the $s$-th moments of $S_k(\alpha;K)$ whenever $K \gg N^{\theta_{k,s}+\epsilon}$ for some $\theta_{k,s}<1/2$. As an immediate consequence, we obtain a lower bound for the $L^1$-mean of the M\"obius-twisted exponential sum over short intervals of length at least $N^{0.49685}$. Moreover, we show that further improvements on all of these results would follow immediately from improvements to an $\ell^2$-estimate involving the M\"obius function.

    \vspace{-1cm}
\end{abstract}

\maketitle

\section{Introduction}

For a positive integer $k \geq 2$, let $\mu_k(n)$ denote the indicator function of the $k$-free integers, that is, the integers  divisible by no $k$th powers of primes. This report is concerned with moments of the short exponential sum
\[S_k(\alpha;K) = \sum_{N-K<n \leq N} \mu_k(n)e(n\alpha)\]
where $K \leq N$ and as usual $e(\alpha) := e^{2\pi i \alpha}$. Most existing results on this sum have been concerned with the case $K=N$. Balog and Ruzsa \cite{balog-ruzsa:1998}, following work of Br\"udern et al. \cite{brudern-etal:1998}, proved that
\[\int_0^1 |S_k(\alpha;K)| \d\alpha \asymp K^{\frac{1}{k+1}}\]
when $K=N$ and the implicit constant may depend only on $k$. This was generalized later by Keil \cite{keil:2013} to the family of estimates
\<\label{keil-bds-non-crit-eq}
\int_0^1 |S_k(\alpha;K)|^s \d\alpha \asymp_{k,s} \begin{cases}
    K^{\frac{s}{k+1}} \quad &\text{if } s<1+\frac{1}{k} \\
    K^{s-1} \quad &\text{if } s>1+\frac{1}{k}
\end{cases}
\>
when $K=N$.

A natural escalation is to determine the persistence of this behavior when $K$ is smaller than $N$. Fixing an integer $k\geq 2$ and a positive real number $s\neq 1+\frac{1}{k}$, we say that an exponent $\theta$ is \emph{$(k,s)$-admissible} if \eqref{keil-bds-non-crit-eq} holds for any $K \gg N^{\theta}$, and we define $\theta_{k,s}$ to be the infimum of all $(k,s)$-admissible $\theta$. Here we allow the implicit constant in \eqref{keil-bds-non-crit-eq} to depend on $k,s,$ and $\theta$. The $s=1$ has been examined by Sun \cite{sun:2023}, who shows that $\theta_{2,1} \leq \frac{18}{29}$ and $\theta_{k,1} \leq \frac{k+1}{2k}$ for $k>2$. Our principal goal is the following improvement and generalization of this result.

\begin{theorem}\label{theta-thm}
    Let $k\geq 2$ and $s \neq 1+\frac{1}{k}$. Then one has that
    \[\theta_{k,s} \leq \begin{cases}
        \big(\frac{k+1}{2}\big)\delta_k \quad &\text{if } s<1+\frac{1}{k}, \\
        \big(\frac{1}{2}+\frac{1}{2(s-1)}\big)\delta_k \quad &\text{if } 1+\frac{1}{k}<s < 2, \\
        \delta_k \quad &\text{if } s\geq 2,
    \end{cases}\]
    where $\delta_2 = \frac{105}{317}$ and for $k \geq 3$,
    \<\label{little-delta-def-eq}
    \delta_k = \frac{1}{k+1}\bigg(1-\frac{3k^3-6k^2-k+4}{3k^{5}+15k^{4}+33k^{3}+21k^{2}+14k+10}\bigg) < \frac{1}{k+1}.
    \>
\end{theorem}

In the particular case $s=1$, Theorem \ref{theta-thm} shows that the expected behavior of the $L^1$-mean occurs in intervals of length shorter than $N^{\frac{1}{2}}$ (though our methods do not encroach much further into this territory; the shortest intervals attained are of length $N^{\frac{1555}{3142}+\epsilon} \approx N^{0.49491}$ in the case $k=4$). The principal improvement which allows the proof of Theorem \ref{theta-thm} to go beyond the square-root barrier is a refined estimate involving the middle part sums
\<\label{c_n-def}
c_n(y,z) := \sum_{\substack{y < d \leq z \\ d^k|n}} \mu(d).
\>
\begin{lemma}\label{c_n-bd-lemma}
    For $k\geq 2$ and $1 \leq y < z\leq N^{\frac{1}{k}}$, one has that
    \<\label{c_n-bd-lemma-eq}
    \sum_{N-K < n \leq N} |c_n(y,z)|^2 \ll Ky^{1-k} + N^\epsilon K^{\frac{1}{k}} + N^{\delta_k+\epsilon},
    \>
    where $\delta_k$ is as defined in \eqref{little-delta-def-eq}.
\end{lemma}
This estimate involves a combination of numerous tools, including applications of bounds of the zeta function and Dirichlet polynomials, the hyperbola method, and both the single- and multi-dimensional van der Corput method. With Lemma \ref{c_n-bd-lemma} in hand, however, it is light work to establish Theorem \ref{theta-thm}. Indeed, one may show that the strength of the bound \eqref{c_n-bd-lemma-eq} is the sole impediment to taking $\theta_{k,s} = 0$. For each $k \geq 2$, let $\Delta_k$ denote the least real number such that for any $\epsilon>0$ and $1 \leq y < z$ one has
    \<\label{Delta-def-eq}
    \sum_{N-K < n \leq N} |c_n(y,z)|^2 \ll Ky^{-1} + N^\epsilon K^{\frac{1}{k}} + N^{\Delta_k+\epsilon}.\>

\begin{theorem}\label{Delta-theorem}
    Let $s \neq 1+\frac{1}{k}$ be a positive real number. Then one has that 
    \[
        \theta_{k,s} \leq \begin{cases}
            \big(\frac{k+1}{2}\big)\Delta_k \quad &\text{if } s<1+\frac{1}{k}, \\
            \big(\frac{1}{2}+\frac{1}{2(s-1)}\big)\Delta_k\quad &\text{if } 1+\frac{1}{k}<s < 2, \\
            \Delta_k \quad &\text{if } s\geq 2.
        \end{cases}
    \]
\end{theorem}

\noindent This is a useful observation, and not entirely trivial. In particular, the lower bound case when $s\geq 2$ requires a more careful argument involving to show the mixed moment bound of Lemma \ref{hH-mixed-moment-lemma}. Thus further improvements to Theorem \ref{theta-thm} follow immediately from improved bounds of the form \eqref{Delta-def-eq}.

\begin{remark}\label{pandey-remark}
    By orthogonality, the special case $s=2$ of our investigation equates to studying the prolifery of the estimate
    \<\label{(s=2)-asymp-eq}
    \#\{n \in (N-K,N]:\; n \;\text{is } k\text{-free}\} \asymp K
    \>
    when $K=N^{\theta}$ for $\theta<1$. Our definitions then have that the exponent $\theta_{k,2}$ is the minimal exponent for which \eqref{(s=2)-asymp-eq} holds whenever $K \gg N^{\theta_{k,2}+\epsilon}$. This is a classical problem, and considerably stronger results exist for this special case: Filaseta and Trifonov \cite{filaseta-trifonov:1996} have shown that $\theta_{k,2} \leq \frac{1}{2k+1}$, and in the case $k=2$ Pandey \cite{pandey:2024} has recently improved this to $\theta_{2,2} \leq \frac{1}{5}-\eta$ for some $\eta>0$ which is not made explicit. 
    
    One motivation for the results on $\theta_{k,2}$ is the problem of bounding gaps between $k$-free integers. It therefore follows from Theorem \ref{Delta-theorem} that a bound of the form \eqref{Delta-def-eq} with $\Delta_2 < \frac{1}{5}-\eta$ would directly imply better gap results, though the bounds proven in here are still quite far from such an improvement.
\end{remark}

The case $s=1+\frac{1}{k}$ has hitherto been absent in our discussion. Keil \cite{keil:2013} has proven the bounds
\<\label{keil-crit-bd-eq}
N^{\frac{1}{k}}\log N \ll \int_0^1 |S_k(\alpha;N)|^{1+\frac{1}{k}} \d\alpha \ll N^{\frac{1}{k}}(\log N)^2
\>
for all $k$, and it is suspected that the lower bound is the correct order of magnitude; the analogous lower bound in the function field setting is known to be accurate by recent work of the author \cite{doyle:2026}. We prove the following short interval result, which in the case $K=N$ improves mildly on the upper bound in \eqref{keil-crit-bd-eq}.

\begin{theorem}\label{crit-thm}
    Let $\Delta_k$ be defined as in \eqref{Delta-def-eq}. For $\epsilon>0$ and for $K \gg N^{(\frac{k+1}{2})\Delta_k+\epsilon}$, one has that
    \<\label{crit-thm-eq}
    K^{\frac{1}{k}}\log K \ll \int_0^1 |S_k(\alpha;K)|^{1+\frac{1}{k}} \d\alpha \ll K^{\frac{1}{k}}(\log K)^{2-\frac{1}{k}}.
    \>
\end{theorem}
\begin{corollary}\label{crit-corollary}
    For $k \geq 2$ and $\epsilon>0$, the equation \eqref{crit-thm-eq} holds for all $K \gg N^{(\frac{k+1}{2})\delta_k+\epsilon}$, where $\delta_k$ is as defined in \eqref{little-delta-def-eq}.
\end{corollary}

We refrain from labeling the exponent $\big(\frac{k+1}{2}\big)\delta_k$ as $(k,1+\frac{1}{k})$-admissible, since we do not know the correct order of magnitude to expect for this moment.

Subconvexity results of this sort are ripe for application to problems involving character sums and averages of arithmetic functions, as discussed in the case $K=N$ by Wooley \cite{wooley:2026}. We mention here only a single example, which was a motivating factor in the original work of Balog and Ruzsa \cite{balog-ruzsa:2001}.

\begin{theorem}
    Let $\epsilon>0$. One has that
    \<\label{mobius-lwr-bd-eq}
    \int_0^1 \bigg|\sum_{N-K < n \leq N} \mu(n)e(n\alpha)\bigg|\d\alpha \gg K^{\frac{1}{6}}
    \>
    whenever $K \gg N^{(\frac{3}{2})\Delta_2+\epsilon}$, where $\Delta_2$ is defined by \eqref{Delta-def-eq}.
\end{theorem}
\begin{proof}
    The proof is entirely the same as that of the complete sum; see Theorem 3 of \cite{balog-ruzsa:2001}. Write
    \[G(\alpha) = \sum_{N-K<n \leq N} \mu(n)e(n\alpha).\]
    We have
    \[\int_0^1 \overline{G(\beta)}G(\alpha+\beta)\d\beta = \sum_{N-K<n \leq N} \mu^2(n)e(n\alpha) = S_2(\alpha;K).\]
    It follows by Theorem \ref{Delta-theorem} that
    \[K^{\frac{1}{3}} \ll \int_0^1 |S_2(\alpha;K)| \d\alpha \leq \int_0^1 \int_0^1 |G(\beta)G(\alpha+\beta)|\d\beta\d\alpha = \bigg(\int_0^1 |G(\alpha)|\d\alpha\bigg)^2. \qedhere\]
\end{proof}

\begin{corollary}\label{mobius-corollary}
    For any $\epsilon>0$, the bound \eqref{mobius-lwr-bd-eq} holds whenever $K \gg N^{\frac{105}{317}+\epsilon} \approx N^{0.49685}$.
\end{corollary}
\noindent This was shown by Sun \cite{sun:2023} when $K \gg N^{\frac{9}{17}+\epsilon}$; our improvement is simply due to the stronger bound in Lemma \ref{c_n-bd-lemma}. Two other results are worth noting here:
\begin{itemize}
    \item Prior to \cite{balog-ruzsa:2001}, Balog and Ruzsa \cite{balog-ruzsa:1998} proved the weaker lower bound of $K^{\frac{1}{8}}(\log K)^{-1}$ in the case $K=N$. Sun \cite{sun:2023} observes that this lower bound holds whenever $K \gg N^{\theta_{2,2}+\epsilon}$, and so by the work of Pandey \cite{pandey:2024} we may take $K \gg N^{\frac{1}{5}-\eta}$ (see Remark \ref{pandey-remark}).

    \item The stronger lower bound of $K^{\frac{1}{4}-\epsilon}$ holds for the $L^1$-mean in the case $K=N$ due to Pandey and Radziwi\l\l \;\cite{pandey-radziwill:2023}, by different methods.
\end{itemize}

We begin in \S 2 by introducing an array of preliminary results, most of which are standard. We also introduce some of the machinery used in our arguments, notably a modified Fej\'er kernel and some standard estimates involving the middle part $c_n(y,z)$. In \S3, we handle the subcritical case $s<1+\frac{1}{k}$ of Theorem \ref{Delta-theorem}. In \S4, we tackle the higher moments, giving the proof of Theorem \ref{crit-thm} and completing the proof of Theorem \ref{Delta-theorem}. Finally, \S5 details the proof of Lemma \ref{c_n-bd-lemma}, which then implies the explicit bounds of Theorem \ref{theta-thm} and Corollary \ref{crit-corollary} by Theorems \ref{Delta-theorem} and \ref{crit-thm}, respectively. This section is the most technical of the paper, but much of the proof is similar to the proof of Lemma 3.1 in \cite{sun:2023}.

Throughout, we use $\ll$ and $\gg$ to denote Vinogradov's notation, with implicit constants depending at most on $k$, $s$, $\theta$, and $\epsilon$, unless otherwise stated. We will write $f \asymp g$ when $f \ll g$ and $g \ll f$. For a real number $\theta$, we define $\|\theta\| = \min\{|\theta-t|:t \in \Z\}$ and $\{\theta\} = \theta - \lfloor \theta\rfloor$. Finally, for a real number $A$ we make regular use of the unconventional notation $\hA = 2^A$, in order to make legible the numerous dyadic dissections cluttering our arguments.

\section{Necessary Tools and Preliminary Lemmata}

There are quite a few results of which we make use, which we list here. Many of these results are classical, and so we refer the reader to texts in which proofs may be found.

\begin{lemma}\label{pw-zeta-bd-lemma}
    Fix $\epsilon>0$. Suppose that $\frac{1}{2} \leq \sigma \leq 1+\epsilon$ and $t \geq 1$. Then one has that
    \[\zeta(\sigma+it) = O(t^{\frac{1}{3}(1-\sigma)+\epsilon}).\]
\end{lemma}
\begin{proof}
    See Chapter 5 of \cite{titchmarsh-vol2:1986}, namely the discussion of convexity of $\mu(\sigma)$ along with Theorem 5.5.
\end{proof}

\begin{lemma}\label{L2-zeta-bd-lemma}
    For $T \geq 2$, one has that
    \[\int_{-T}^T |\zeta(1/2+it)|^2 \d t \ll T\log T.\]
\end{lemma}
\begin{proof}
    This follows from Theorem 7.2(A) of \cite{titchmarsh-vol2:1986}.
\end{proof}

\begin{lemma}[\cite{iwaniec-kowalski:2004}, Theorem 9.1]\label{dirichlet-poly-bd-lemma}
    For $T,N \geq 2$ and any sequence of complex numbers $a_n$, one has that
    \[\int_0^T \bigg|\sum_{1 \leq n \leq N} a_nn^{it}\bigg|^2 \d t = (T+O(N))\sum_{1 \leq n \leq N} |a_n|^2.\]
\end{lemma}

\begin{lemma}\label{Kusmin-Landau-lemma}
    Let $f(x)$ be a real function with $\delta \leq |f'(x)| \leq 1-\delta$ and $f''(x) \neq 0$ on $[a,b]$. Then
    \[\sum_{a < n < b} e(f(n)) \ll \delta^{-1},\]
    where the implicit constant is independent of $f$.
\end{lemma}
\begin{proof}
    This is the case $g(x) = 1$ of Corollary 8.11 in \cite{iwaniec-kowalski:2004}.
\end{proof}

We define the sawtooth function $\psi$ by
\[\psi(x) = \begin{cases}
    \{x\}-\frac{1}{2} \quad &\text{if } x \not \in \Z, \\
    0 \quad &\text{if } x \in \Z.
\end{cases}\]

\begin{lemma}\label{sawtooth-lemma}
    Fix $R>0$. There exist coefficients $\beta(r) \ll |r|^{-1}$ such that
    \[\psi(x) \leq R^{-1} + \sum_{1 \leq |r| \leq R} \beta(r)e(rx).\]
\end{lemma}
\begin{proof}
    See Theorem A.6 of \cite{graham-kolesnik:1991}.
\end{proof}

We define an exponent pair to be a pair $0 \leq p,q \leq \frac{1}{2}$ such that, whenever $\epsilon>0$, one has
\<\label{exp-pair-def-eq}
\sum_{a \leq n \leq b} e(f(n)) \ll |f'(x)|^p M^{q+\frac{1}{2}}F^\epsilon
\>
for any $[a,b] \subseteq [M,2M]$ and for any smooth function $f$ on $[M,2M]$ such that $|f(x)| \asymp F$ and $|f'(x)| \asymp FM^{-1}$, as in \cite{iwaniec-kowalski:2004}.

\begin{lemma}[\cite{sun:2023}, Lemma 2.5]\label{sv-vdc-bd-lemma}
    Let $M > 0$. Suppose that $(p,q)$ is an exponent pair as in \eqref{exp-pair-def-eq} and that $f(x)$ satisfies
    \[|f^{(j)}(x)| \asymp_j FM^{-j}\]
    for all $x \in [M,2M]$, every $j \geq 0$ and for some $F \geq M$. Then for any interval $I \subseteq [M,2M]$, one has that
    \[\sum_{m \in I} \psi(f(m)) \ll F^{\frac{p}{p+1}}M^{\frac{1+2q}{2(p+1)}+\epsilon}\]
    for any $\epsilon>0$.
\end{lemma}

\begin{theorem}\label{vdc2-ref-lemma}
    Let $\epsilon>0$. Suppose $f(x,y) = Ax^{-\alpha}y^{-\beta}$, where $A,\alpha,$ and $\beta$ are real numbers with $\alpha,\beta>0$. Let $D = (X,2X] \times (Y,2Y]$, and write $F = AX^{-\alpha}Y^{-\beta}$ and $N=XY$. Then one has that
    \[\sum_{(m,n) \in D} e(f(m,n)) \ll F^{1/3}N^{1/2} + F^\epsilon N^{5/6+\epsilon} + F^{\epsilon-1/8}N^{15/16+\epsilon} + F^{\epsilon-1/4}N^{1+\epsilon}.\]
\end{theorem}
\begin{proof}
    This is a weakened version of Theorem 6.12 of \cite{graham-kolesnik:1991}.
\end{proof}

In the proof of the lower bound in the subcritical case, we will also make use of a modified Fej\'er kernel given by
\<\label{Fejer-kernel-def-eq}
F(\alpha) = e\bigg(\bigg(N-\frac{K}{2}\bigg)\alpha\bigg)\sum_{|h| \leq K} \bigg(1-\dfrac{h}{K}\bigg)e(h\alpha) = \sum_{N-K < n \leq N} \bigg(1-\dfrac{|2N-K-2n|}{K}\bigg)e(n\alpha).
\>
One may see that
\<\label{F(a)-sin-ident-eq}
F(\alpha) = \frac{1}{K}e\bigg(\bigg(N-\frac{K}{2}\bigg)\alpha\bigg)\dfrac{\sin^2(\pi K\alpha)}{\sin^2(\pi\alpha)} \ll \min\bigg\{K,\frac{1}{K\|\alpha\|^{2}}\bigg\}.
\>
Using the standard Fourier convolution
\[(f * g)(\alpha) = \int_0^1 f(\beta)g(\alpha-\beta)\d\beta\]
we may also then deduce that
\<\label{lower-bd-smoothing-ineq}
\int_0^1 |(F * S_k)(\alpha)| \d\alpha \ll \int_0^1 \int_0^1 |F(\alpha-\beta)||S_k(\beta;K)|\d\beta\d\alpha \leq \int_0^1 |S_k(\alpha;K)|\d\alpha.
\>
Thus to prove the lower bound on the $L^1$-mean, it is sufficient to prove it for the Fej\'er-smoothed sum
\[(F*S_k)(\alpha) = \sum_{N-K < n \leq N} \bigg(1-\dfrac{|2N-K-2n|}{K}\bigg)\mu_k(n)e(n\alpha).\]

Finally, we introduce the notation by which the middle part estimate of Lemma \ref{c_n-bd-lemma} is injected into our arguments. Observe the standard identity
\<\label{mu_k-conv-id-eq}
\mu_k(n) = \sum_{d^k|n} \mu(d).
\>
Recalling the definition \eqref{c_n-def}, we write
\<\label{T_i-def-eq}
T_i(\alpha) = T_i(\alpha;K) = \sum_{N-K<n\leq N} c_n(\hi,2\hi)e(n\alpha),
\>
so that $S_k(\alpha;K) = \sum T_i(\alpha)$. We will also make use of longer segments in order to refine the argument, defining
\<\label{H,h-def-eq}
h_j(\alpha) = h_j(\alpha;K) = \sum_{i< j} T_i(\alpha), \qquad H_j(\alpha) = H_j(\alpha;K) = \sum_{i\geq j} T_i(\alpha).
\>
A number of estimates regarding these exponential sums will be essential for our arguments.

\begin{lemma}\label{T_i-bounds-lemma}
    One has that, for $s>1$,
    \[
        \int_0^1 |T_i(\alpha)| \d\alpha \ll \hi \log K \qquad \text{and} \qquad \int_0^1 |T_i(\alpha)|^s \d\alpha \ll \hi K^{s-1}.
    \]
\end{lemma}
\begin{proof}
    The proof is the same as in the case $K=N$; see Lemma 2.2 of \cite{keil:2013}. We handle both bounds together. Applying the definitions \eqref{c_n-def} and \eqref{T_i-def-eq}, we may swap the order of summation and apply the triangle inequality to write 
    \<\label{T_i-lemma-proof-eq-1}
    \int_0^1 |T_i(\alpha)|^s \d\alpha \leq \int_0^1 \bigg(\sum_{\hi < d \leq 2\hi} \bigg|\sum_{\frac{N-K}{d^k} < n \leq \frac{N}{d^k}} e(nd^k\alpha)\bigg|\bigg)^s\d\alpha.
    \>
    If $s=1$, then this is bounded by
    \[\sum_{\hi < d \leq 2\hi} \int_0^1 \bigg|\sum_{\frac{N-K}{d^k} < n \leq \frac{N}{d^k}} e(nd^k\alpha)\bigg| \d\alpha.\]
    The change of variable $\beta = d^k\alpha$ and the bound 
    \<\label{lin-exp-sum-bd}
    \sum_{1 \leq n \leq K/d^k} e(n\alpha) \ll \min\bigg\{\frac{K}{d^k},\frac{1}{\|\alpha\|}\bigg\}\>
    thus results in the bound
    \[\int_0^1 |T_i(\alpha)| \d\alpha \ll \sum_{\hi < d \leq 2\hi} \log(K) \ll \hi \log K.\]

    On the other hand, if $s>1$, we may apply H\"older's inequality to \eqref{T_i-lemma-proof-eq-1} to see that
    \[\int_0^1 |T_i(\alpha)|^s \d\alpha \ll \bigg(\sum_{\hi < d \leq 2\hi} 1^{\frac{s}{s-1}}\bigg)^{s-1}\sum_{\hi < d \leq 2\hi}\int_0^1 \bigg|\sum_{\frac{N-K}{d^k} < n \leq \frac{N}{d^k}} e(nd^k\alpha)\bigg|^s \d\alpha.\]
    The first sum is of size $\hi^{s-1}$, while the remaining integral may again be handled by the change of variable $\beta = d^k\alpha$ and the bound \eqref{lin-exp-sum-bd}, so that we may crudely bound it as
    \[\int_0^1 |T_i(\alpha)|^s \d\alpha \ll \hi^{s-1} \sum_{\hi < d \leq 2\hi} \int_0^1 \Bigg(\min\bigg\{\frac{K}{d^k},\frac{1}{\|\beta\|}\bigg\}\Bigg)^s \d\beta \ll \hi K^{s-1}. \qedhere\]
\end{proof}

\begin{corollary}\label{h_i-bound-lemma}
    One has that
    \[\int_0^1 |h_i(\alpha)| \d\alpha \ll \hi \log K.\]
\end{corollary}
\begin{proof}
    Apply the triangle inequality and Lemma \ref{T_i-bounds-lemma} to the definition of $h_i(\alpha)$ in \eqref{H,h-def-eq}.
\end{proof}

As a consequence of Parseval's identity and \eqref{Delta-def-eq}, we may also deduce the following $L^2$-mean estimates.

\begin{corollary}\label{L2-T_i,H_i-bounds-lemma}
    For any $\epsilon>0$, one has that
    \[\int_0^1 |T_i(\alpha)|^2 \d\alpha  \ll K\hi^{1-k} + N^\epsilon K^{\frac{1}{k}} + N^{\Delta_k+\epsilon} \] 
    and
    \[\int_0^1 |H_i(\alpha)|^2 \d\alpha \ll K\hi^{1-k} + N^{\epsilon}K^{\frac{1}{k}} + N^{\Delta_k+\epsilon}.\]
\end{corollary}
It is an important observation that the term $K\hi^{1-k}$ dominates this bound whenever one has $\hi \ll \min\{K^{\frac{1}{k}-\epsilon},(KN^{-\Delta_k})^{\frac{1}{k-1}}\}$. In particular, if $K \gg N^{(\frac{k+1}{2})\Delta_k+\epsilon}$, then this occurs whenever $\hi \ll K^{\frac{1}{k+1}}$. One may also deduce from \eqref{Delta-def-eq} an immediate pointwise estimate on $H_i(\alpha)$.
\begin{lemma}\label{H_i-pw-lemma}
    For any $\alpha \in \R$ and any $\epsilon > 0$, one has that
    \[|H_i(\alpha)| \ll K\hi^{\frac{1-k}{2}} + N^\epsilon K^{\frac{k+1}{2k}} + K^{\frac{1}{2}}N^{\frac{\Delta_k}{2}+\epsilon}.\]
\end{lemma}
\begin{proof}
    This follows easily from the Cauchy-Schwarz inequality. We have that
    \[H_i(\alpha) \leq K^{\frac{1}{2}}\bigg(\sum_{N-K < n \leq N} |c_n(\hi,N^{\frac{1}{k}})|^2\bigg)^{\frac{1}{2}},\]
    so the definition \eqref{Delta-def-eq} reveals the bound.
\end{proof}

Finally, some very careful interpolation of these estimates yields a mixed moment bound which is vital for the analysis of the case $s=2$.

\begin{lemma}\label{hH-mixed-moment-lemma}
    Let $\epsilon>0$. For $K \gg N^{\Delta_k+\epsilon}$, one has that
    \[\int_0^1 |h_D(\alpha)H_D(\alpha)| \d\alpha = o(K).\]
\end{lemma}
\begin{proof}
    It is enough to show the result for $\epsilon$ sufficiently small. Observe first that we may choose some $C = c\log(N)$, where $c<\frac{1}{k+1}$ will be determined later, and write
    \<\label{s=2-decomp-eq}
    \int_0^1 |h_D(\alpha)H_D(\alpha)| \d\alpha \ll \int_0^1 |h_C(\alpha)H_D(\alpha)| \d\alpha + \int_0^1 \bigg|\sum_{C<i<D} T_i(\alpha)H_D(\alpha)\bigg|\d\alpha.
    \>
    For the first integral, we write 
    \[\int_0^1 |h_C(\alpha)H_D(\alpha)| \d\alpha \ll \bigg(\sup_\alpha |H_D(\alpha)|\bigg) \int_0^1 |h_C(\alpha)|\d\alpha\]
    and apply the bounds of Corollary \ref{h_i-bound-lemma} and Lemma \ref{H_i-pw-lemma} to see that, for any $\epsilon'>0$, one has that
    \[\int_0^1 |h_C(\alpha)H_D(\alpha)|\d\alpha \ll (K^{\frac{k+3}{2(k+1)}} + N^{\epsilon'} K^{\frac{k+1}{2k}} + K^{\frac{1}{2}}N^{\frac{\Delta_k}{2}+\epsilon'})\hC\log(N).\]
    If one sets $c = \epsilon' = \frac{1}{100}\epsilon$, say, then it is clear that each term in the parentheses is $o(K^{1-2\epsilon'})$, so that
    \<\label{s=2-eq-1}
    \int_0^1 |h_C(\alpha)H_D(\alpha)|\d\alpha \ll (K^{1-2\epsilon'})\hC \log N = o(K).
    \>
    
    For the remaining integral, first write
    \[\int_0^1 \bigg|\sum_{C<i<D} T_i(\alpha)H_D(\alpha)\bigg|\d\alpha \leq \sum_{C<i<D}\int_0^1 |T_i(\alpha)H_D(\alpha)|\d\alpha.\]
    We apply H\"older's inequality to see that
    \begin{multline*}
        \int_0^1 \bigg|\sum_{C<i<D} T_i(\alpha)H_D(\alpha)\bigg|\d\alpha \leq \sum_{C<i<D} \bigg(\sup_\alpha |H_D(\alpha)|\bigg)^{\frac{1}{i}}\bigg(\int_0^1 |T_i(\alpha)|^2\d\alpha\bigg)^{\frac{i-1}{2i}} \\
        \times \bigg(\int_0^1 |T_i(\alpha)|\d\alpha\bigg)^{\frac{1}{i}}\bigg(\int_0^1|H_D(\alpha)|^2\d\alpha\bigg)^{\frac{i-1}{2i}}.
    \end{multline*}
    Now we apply the bounds of Lemmata \ref{T_i-bounds-lemma} and \ref{H_i-pw-lemma} and of Corollary \ref{L2-T_i,H_i-bounds-lemma} to see that, for any $\epsilon''>0$, the integral is bounded by
    \[(\log N) \sum_{C<i<D}  (K^{\frac{k+3}{2(k+1)}} + K^{\frac{1}{2}}N^{\frac{\Delta_k}{2}+\epsilon''})^{\frac{1}{i}}(K\hi^{1-k} + N^{\Delta_k+\epsilon''})^{\frac{i-1}{2i}}(K^{\frac{2}{k+1}} + N^{\Delta_k+\epsilon''})^{\frac{i-1}{2i}}.\]
    Upon choosing $\epsilon''$ small enough (say $\epsilon'' = \frac{1}{100}\epsilon$ again), each term in the parentheses is of order no more than $K^{1-\epsilon''}$, and so we have that
    \<\label{s=2-eq-2}
    \int_0^1 \bigg|\sum_{C<i<D} T_i(\alpha)H_D(\alpha)\bigg|\d\alpha \ll K^{1-\epsilon''}\log N \sum_{C<i<D} \hi^{\frac{1}{i}} \ll K^{1-\epsilon''}(\log N)^2 = o(K).
    \>
    Applying \eqref{s=2-eq-1} and \eqref{s=2-eq-2} to \eqref{s=2-decomp-eq}, we have the result.
\end{proof}

\section{Subcritical Moment Estimates}

The treatment of the subcritical moment estimates is similar to the approach used in \cite{sun:2023}, though we treat general $k$ and $s$. We begin by proving the desired upper bound. As mentioned previously, we prove these estimates in terms of $\Delta_k$, so that Lemma \ref{c_n-bd-lemma} then implies Theorem \ref{theta-thm}.

\begin{lemma}\label{Ls-upper-subcrit-lemma}
    Let $\epsilon>0$. For $s<1+1/k$ and $K \gg N^{(\frac{k+1}{2})\Delta_k+\epsilon}$, we have that
    \[\int_0^1 |S_k(\alpha;K)|^s \d\alpha \ll K^{\frac{s}{k+1}}.\]
\end{lemma}
\begin{proof}
    We assume from now on that $1 < s \leq 1+1/k$; the case $s \leq 1$ follows from an application of H\"older's inequality. We first make the observation that
    \[\int_0^1 |S_k(\alpha;K)|^s \d\alpha \ll \int_0^1 |h_D(\alpha)|^s\d\alpha + \int_0^1 |H_D(\alpha)|^s\d\alpha,\]
    where $D = (k+1)^{-1}\log(K)$. For the second term, H\"older's inequality, Corollary \ref{L2-T_i,H_i-bounds-lemma}, and the definition \eqref{Delta-def-eq} of $\Delta_k$ imply that for some sufficiently small $\epsilon'>0$, we have
    \[\int_0^1 |H_D(\alpha)|^s\d\alpha\ll \bigg(K^{\frac{2}{k+1}} + N^{\epsilon'} K^{\frac{1}{k}} + N^{\Delta_k+\epsilon'}\bigg)^{\frac{s}{2}} \ll K^{\frac{s}{k+1}} + N^{\frac{s\Delta_k}{2}+\epsilon} \ll K^{\frac{s}{k+1}}.\]

    It remains then to bound the first term. We apply a polynomial weight, observing that
    \begin{align*}
        \int_0^1 |h_D(\alpha)|^s\d\alpha &\ll \int_0^1 \bigg|\sum_{i<D} (D-i)^{-1}(D-i)T_i(\alpha)\bigg|^s\d\alpha \\
        &\ll \bigg(\sum_{i<D} (D-i)^{-\frac{s}{s-1}}\bigg)^{s-1}\sum_{i<D} (D-i)^s\int_0^1 |T_i(\alpha)|^s\d\alpha.
    \end{align*}
    Since $s<2$, the first sum over $i$ is $O(1)$, and so this simplifies to
    \<\label{T_i-in-Ls-upper-eq}
    \int_0^1 |h_D(\alpha)|^s\d\alpha \ll \sum_{i<D} (D-i)^s\int_0^1 |T_i(\alpha)|^s\d\alpha.
    \>
    A direct application of the $L^s$-bounds of Lemma \ref{T_i-bounds-lemma} is insufficient for our needs, as the summation introduces a logarithmic factor. Instead, we use the Fej\'er kernel introduced in \eqref{Fejer-kernel-def-eq} to smooth our exponential sum.
    
    We may assume (by permitting an error term of order $O(1)$) that $K$ is an even integer, and we set $M = K^{\frac{1}{k+1}}2^{ik}$. Observe that $M = O(K)$ whenever $i<D$. As in \cite{sun:2023}, we reindex our sum and dissect $T_i(\alpha)$ as 
    \begin{multline*}
        T_i(\alpha) = \sum_{|n| < M + K/2} \min\bigg\{1,\dfrac{2M+K - |2n|}{2M}\bigg\} c_{N-K/2+n}(\hi,2\hi)e((N-K/2+n)\alpha) \\
        - \sum_{K/2 < |n| < M+K/2} \dfrac{2M+K-|2n|}{2M}c_{N-K/2+n}(\hi,2\hi)e((N-K/2+n)\alpha).
    \end{multline*}
    Upon denoting these two sums by $Q_i(\alpha)$ and $R_i(\alpha)$, respectively, we see that \eqref{T_i-in-Ls-upper-eq} may be bounded as
    \<\label{T_i-to_Q_i,R_i-eq}
    \int_0^1 |h_D(\alpha)|^s\d\alpha \ll \sum_{i<D} (D-i)^s\bigg(\int_0^1 |Q_i(\alpha)|^s\d\alpha + \int_0^1 |R_i(\alpha)|^s\d\alpha\bigg).\>
    We have that
    \[
    \int_0^1 |R_i(\alpha)|^s \d\alpha \ll \bigg(\int_0^1 |R_i(\alpha)|^2 \d\alpha\bigg)^{\frac{s}{2}} \ll \bigg(\sum_{n \in I} |c_n(\hi,2\hi)|^2\bigg)^{\frac{s}{2}},
    \]
    where $I = [N-K-M,N-K] \cup [N,N+M]$. Applying \eqref{Delta-def-eq}, then, we see that
    \<\label{R_i-in-Ls-upper-eq}
    \sum_{i<D}(D-i)^s\int_0^1 |R_i(\alpha)|^s \d\alpha \ll \sum_{i<D}(D-i)^s\bigg(M\hi^{1-k} + N^\epsilon M^{\frac{1}{k}} + N^{\Delta_k+\epsilon}\bigg)^{\frac{s}{2}} \ll K^{\frac{s}{k+1}}.
    \>
    
    Thus we may proceed to consideration of $Q_i(\alpha)$. In order to obtain sufficient bounds for $1<s<1+\frac{1}{k}$, we must interpolate first and second moment estimates for $Q_i(\alpha)$. For the first moment, we may apply the definition \eqref{c_n-def} of $c_n(y,z)$, swap the order of summation, and apply the triangle inequality to see that
    \[\int_0^1 |Q_i(\alpha)|\d\alpha \leq \sum_{\hi \leq d \leq 2\hi} \int_0^1 \Bigg|\sum_{\substack{|n| < M + K/2 \\ n \equiv K/2-N \mmod{d^k}}} \min\bigg\{1,\dfrac{2M+K - |2n|}{2M}\bigg\}e(n\alpha)\Bigg|\d\alpha.\]
    This new exponential sum is an approximate Fej\'er kernel, which may easily estimated (see (2.5) in \cite{balog-ruzsa:2001}) to yield
    \begin{align*}
        \int_0^1 |Q_i(\alpha)|\d\alpha &\ll \sum_{\hi \leq d \leq 2\hi} \int_0^1 \Bigg|\min\bigg\{\dfrac{M+K}{d^k},\dfrac{1}{\|d^k\alpha\|},\dfrac{d^k}{M\|d^k\alpha\|^2}\bigg\} + O(1)\Bigg| \d\alpha \\
        &\ll \sum_{\hi \leq d \leq 2\hi} \bigg(\log\bigg(\dfrac{K}{M}\bigg) + 1\bigg) \ll \hi\log(K^{\frac{k}{k+1}}\hi^{-k}).
    \end{align*}

    On the other hand, we have for $\epsilon'>0$ sufficiently small that
    \[\int_0^1 |Q_i(\alpha)|^2 \d\alpha \ll \sum_{N-K-M < n \leq N+M} |c_n(\hi,2\hi)|^2 \ll K\hi^{1-k} + N^{\epsilon'} K^{\frac{1}{k}} + N^{\Delta_k+\epsilon'}\]
    by Lemma \ref{c_n-bd-lemma}. In fact, since $i<D$, we have that $K\hi^{1-k} \gg K\hD^{1-k} = K^{\frac{2}{k+1}} \gg N^{\Delta_k+\epsilon'}$, and so
    \[\int_0^1 |Q_i(\alpha)|^2 \d\alpha \ll K\hi^{1-k}.\] 
    Thus upon interpolating with the $L^1$-bound we have that
    \begin{align*}
        \sum_{i<D}(D-i)^s\int_0^1 |Q_i(\alpha)|^s \d\alpha &\ll \sum_{i<D}(D-i)^s\bigg(\int_0^1 |Q_i(\alpha)| \d\alpha\bigg)^{2-s}\bigg(\int_0^1 |Q_i(\alpha)|^2\d \alpha\bigg)^{s-1} \\
        &\ll \sum_{i<D}(D-i)^s\hi^{2-s}(k(D-i))^{2-s}K^{s-1}\hi^{(1-k)(s-1)} \\
        &\ll K^{s-1}\sum_{i<D}\hi^{(1-k(s-1))}(D-i)^2.
    \end{align*}
    Since $s<1+\frac{1}{k}$, this is $O(K^{\frac{s}{k+1}})$, and so combining this with \eqref{T_i-to_Q_i,R_i-eq} and \eqref{R_i-in-Ls-upper-eq}, we have the desired result.
\end{proof}

The argument for the lower bound similarly is a direct translation of \cite{sun:2023}, the only change being the insertion of our new bound. Because the argument is essentially the same, we leave out the finer details.

\begin{lemma}\label{Ls-lower-subcrit-lemma}
    Let $\epsilon>0$. For $s<1+\frac{1}{k}$ and $K \gg N^{(\frac{k+1}{2})\Delta_k+\epsilon}$, we have that 
    \[\int_0^1 |S_k(\alpha;K)|^s \d\alpha \gg K^{\frac{s}{k+1}}.\]
\end{lemma}
\begin{proof}
    It is sufficient via H\"older convexity to prove this result for $s=1$. To see this, suppose the result holds for $s=1$, and observe that
    \[\int_0^1 |S_k(\alpha;K)|^s \d\alpha \gg \bigg(\int_0^1 |S_k(\alpha;K)| \d\alpha\bigg)^s \gg K^{\frac{s}{k+1}}\]
    when $s>1$, and for $s < 1$ we may apply H\"older's inequality to obtain
    \[N^{\frac{1}{k+1}} \ll \int_0^1 |S_k(\alpha;K)| \d\alpha \ll \bigg(\int_0^1 |S_k(\alpha;K)|^s \d\alpha\bigg)^{\frac{\delta}{1+\delta-s}}\bigg(\int_0^1 |S_k(\alpha;K)|^{1+\delta} \d\alpha\bigg)^{\frac{1-s}{1+\delta-s}},\]
    for any $0<\delta<\frac{1}{k}$, which then implies the result by an application of Lemma \ref{Ls-upper-subcrit-lemma}.

    Thus we restrict to the case $s=1$. As observed in \eqref{lower-bd-smoothing-ineq}, we may consider instead the smoothed sum $(F*S_k)(\alpha)$. For $q \in \N$, define
    \[
    F_q(\alpha) = q^{-1}\sum_{a=1}^q F\bigg(\alpha - \dfrac{a}{q}\bigg) = \sum_{\substack{N-K < n \leq N \\ n \equiv 0 \mmod{q}}} \bigg(1-\dfrac{|2N-K-2n|}{K}\bigg)e(n\alpha).
    \]
    Then using the identity \eqref{mu_k-conv-id-eq} we may write
    \begin{align*}
        (F*S_k)(\alpha) &= \sum_{N-K < n \leq N} \bigg(1-\dfrac{|2N-K-2n|}{K}\bigg)\mu_k(n)e(n\alpha) \\
        &= \sum_{d \leq K^{1/(k+1)}} \mu(d)F_{d^k}(\alpha) + \sum_{N-K < n \leq N} \bigg(1-\dfrac{|2N-K-2n|}{K}\bigg) c_n(K^{\frac{1}{k+1}},N^{\frac{1}{k}}) e(n\alpha),
    \end{align*}
    Our intention is to restrict the region of integration to some region $\calY \subset [0,1]$, to be defined later, on which we can ensure the first sum is large. With this in mind, we can dispose of the second sum, applying the Cauchy-Schwarz inequality, Parseval's identity, and the definition \eqref{Delta-def-eq} of $\Delta_k$ to yield
    \<\label{L1-lower-decomp-1-eq}
        \int_0^1 (F * S_k)(\alpha) \d\alpha \gg \int_\calY \bigg|\sum_{d \leq K^{1/(k+1)}} \mu(d)F_{d^k}(\alpha)\bigg| \d\alpha - |\calY|^{\frac{1}{2}}K^{\frac{1}{k+1}},
    \>
    provided $K \gg N^{\frac{\Delta_k(k+1)}{2}+\epsilon}.$

    Now we approach the first term. We begin by expanding the definition of $F_{d^k}(\alpha)$ and sorting according to $m|d$ to obtain
    \[\sum_{d \leq K^{1/(k+1)}} \mu(d)F_{d^k}(\alpha) = \sum_{d \leq K^{1/(k+1)}} \sum_{m|d} \sum_{\substack{1 \leq a \leq d^k \\ m^k|a}}\dfrac{\mu(d)}{d^k}\mu_k\bigg(\bigg(\frac{a}{m^k},\frac{d^k}{m^k}\bigg)\bigg)F\bigg(\alpha-\dfrac{a}{d^k}\bigg).\]
    Making the change of variables $d = um$ and $a = vm^k$, we may rewrite this as
    \<\label{L1-lower-first-int-decomp}
    \sum_{d \leq K^{1/(k+1)}} \mu(d)F_{d^k}(\alpha) = \sum_{u \leq K^{1/(k+1)}} \mu(d) b_d G_d(\alpha),
    \>
    where
    \<\label{b_d-ineq}
    b_d = \sum_{\substack{m \leq K^{1/(k+1)}/d \\ (m,d) = 1}} \dfrac{\mu(m)}{m^k}
    \>
    and
    \[G_d(\alpha) = \dfrac{1}{d^k} \sum_{\substack{1 \leq a \leq d^k \\ (a,d^k) \;k\text{-free}}} F\bigg(\alpha-\dfrac{a}{d^k}\bigg).\]
    One may immediately see that
    \[\frac{1}{3} \leq 1-(\zeta(k)-1) \leq 1-\sum_{d \geq 2} d^{-k} \leq b_d \leq 1+\sum_{d \geq 2} d^{-k} \leq \dfrac{\pi^2}{6} \leq \dfrac{5}{3},\]
    so that $b_d$ is essentially constant. 
    
    The rest of the argument has no novelty and is somewhat involved, and so we refer the reader to \cite{sun:2023} for a more detailed treatment. In particular, we define
    \[\calX_d = \bigcup_{\substack{1 \leq a \leq d^k \\ (a,d^k) \; k\text{-free}}} \bigg[\frac{a}{d^k}-\frac{1}{2K},\frac{a}{d^k}+\frac{1}{2K}\bigg].\]
    From this, we define $\calY_d$ to be the subset of $\calX_d$ in which 
    \[\sum_{\substack{d' \leq K^{1/(k+1)} \\ d' \neq d}} |G_{d'}(\alpha)| \leq \frac{K}{20 d^k}.\]
    While the sets $\calX_d$ are not necessarily disjoint, Lemma 4.2 of \cite{sun:2023} shows that the $\calY_d$ are disjoint. We also have by Lemma 4.1 of \cite{sun:2023} that for $\alpha \in \calX_d$, we have
    \[|G_d(\alpha)| \geq \frac{K}{2d^k}(1+o(1)).\]
    Finally, the arguments of Lemma 4.3 in \cite{sun:2023} allow us to conclude that $|\calY_d| \gg |\calX_d|$ whenever $d \leq \epsilon' K^{\frac{1}{k+1}}$ for some sufficiently small $\epsilon'>0$.
    
    As a result, we may conclude that, upon fixing some sufficiently small $\epsilon'>0$ and setting $\calY = \bigcup_{d \leq \epsilon' K^{1/(k+1)}} \calY_d$, we may return to the integral in \eqref{L1-lower-decomp-1-eq} and apply \eqref{L1-lower-first-int-decomp}, \eqref{b_d-ineq}, and the results of the previous paragraph to see that
    \<\label{subcrit-lower-Y-int-eq}
    \int_\calY \bigg|\sum_{d \leq K^{1/(k+1)}} \mu(d)F_{d^k}(\alpha)\bigg|\d\alpha \gg \sum_{d \leq \epsilon' K^{1/(k+1)}} \dfrac{K}{d^k}|\calY_d| \gg \epsilon' K^{\frac{1}{k+1}}.
    \>
    On the other hand, $|\calY| \ll \sum_{d \leq \epsilon' K^{1/(k+1)}}|\calX_d| \leq (\epsilon')^{k+1}$, and so 
    \<\label{subcrit-lower-term2-bd-eq}
    |\calY|^{\frac{1}{2}}K^{\frac{1}{k+1}} \ll (\epsilon')^{\frac{k+1}{2}}K^{\frac{1}{k+1}}.
    \>
    Thus we have the desired result by \eqref{lower-bd-smoothing-ineq}, \eqref{L1-lower-decomp-1-eq}, \eqref{subcrit-lower-Y-int-eq}, and \eqref{subcrit-lower-term2-bd-eq}.
\end{proof}

Lemmata \ref{Ls-upper-subcrit-lemma} and \ref{Ls-lower-subcrit-lemma} together imply the case $s<1+\frac{1}{k}$ of Theorem \ref{Delta-theorem}.

\section{Higher Moment Estimates}

The higher moments are treated with little distinction from their complete interval counterparts in \cite{keil:2013}, though there are some important new considerations we must make. We start with a mild refinement of an argument in \cite{keil:2013}, introduced by the author in the function field setting \cite{doyle:2026}, which allows us to trim some of the excess in the upper bound for $s=1+\frac{1}{k}$.

\begin{lemma}\label{crit-upr-lemma}
    Let $\epsilon>0$. For $K \gg N^{(\frac{k+1}{2})\Delta_k+\epsilon}$, one has that
    \[\int_0^1 |S_k(\alpha;K)|^{1+\frac{1}{k}} \d\alpha \ll K^{\frac{1}{k}}(\log K)^{2-\frac{1}{k}}.\]
\end{lemma}
\begin{proof}
    We begin by dissecting $S_k(\alpha;K)$ as in the previous section, writing
    \<\label{crit-decomp-eq-1}
    \int_0^1 |S_k(\alpha;K)|^{1+\frac{1}{k}} \d\alpha \ll \int_0^1 |h_D(\alpha)|^{1+\frac{1}{k}} \d\alpha + \int_0^1 |H_D(\alpha)|^{1+\frac{1}{k}} \d\alpha,
    \>
    where as before we have that $D = (k+1)^{-1}\log(K)$. Applying H\"older's inequality and Corollary \ref{L2-T_i,H_i-bounds-lemma}, we have upon choosing $\epsilon'>0$ small enough that
    \[\int_0^1 |H_D(\alpha)|^{1+\frac{1}{k}} \d\alpha \ll (K\hD^{1-k} + N^{\epsilon'} K^{\frac{1}{k}} + N^{\Delta_k+\epsilon'})^{\frac{k+1}{2k}} \ll K^{\frac{1}{k}}\]
    by the assumption $K \gg N^{(\frac{k+1}{2})\Delta_k+\epsilon}$.

    We turn to the treatment of the first term in \eqref{crit-decomp-eq-1}. By an application of the triangle inequality, we may write
    \[\int_0^1 |S_k(\alpha;K)|^{1+\frac{1}{k}} \d\alpha \ll \log(X)\max_{j<D}\bigg( \int_0^1 |T_j(\alpha)|\bigg|h_j(\alpha)+\sum_{j\leq i<D}T_i(\alpha)\bigg|^{\frac{1}{k}} \d\alpha\bigg).\]
    An application of the bound $|x+y|^p \ll |x|^p + |y|^p$ then shows that
    \<\label{crit-upper-decomp-eq-2}
    \int_0^1 |S_k(\alpha;K)|^{1+\frac{1}{k}}\d\alpha \ll \log(X)\max_{j<D} \bigg(V_1(j)+V_2(j)\bigg),\>
    where
    \[
    V_1(j) = \int_0^1 |T_j(\alpha)||h_j(\alpha)|^{\frac{1}{k}}\d\alpha \quad \text{and} \quad V_2(j) = \int_0^1|T_j(\alpha)|\bigg|\sum_{j\leq i<D}T_i(\alpha)\bigg|^{\frac{1}{k}} \d\alpha.
    \]
    Precise usage of H\"older's inequality then yield
    \[V_1(j) \ll \bigg(\int_0^1 |T_j(\alpha)|^2\d\alpha\bigg)^{\frac{1}{k}}\bigg(\int_0^1 |T_j(\alpha)|\d\alpha\bigg)^{1-\frac{2}{k}}\bigg(\int_0^1|h_j(\alpha)|\d\alpha\bigg)^{\frac{1}{k}}\]
    and
    \[V_2(j) \ll \bigg(\int_0^1 |T_j(\alpha)|^2\d\alpha\bigg)^{\frac{1}{2k}}\bigg(\int_\T |T_j(\alpha)|\d\alpha\bigg)^{1-\frac{1}{k}}\bigg(\int_0^1\bigg|\sum_{j\leq i<D}T_i(\alpha)\bigg|^{2}d\alpha\bigg)^{\frac{1}{2k}}.\]
    Lemmata \ref{T_i-bounds-lemma} and \ref{h_i-bound-lemma} and Corollary \ref{L2-T_i,H_i-bounds-lemma} then beget the estimate
    \[
    \max\{V_1(j),V_2(j)\} \ll K^{\frac{1}{k}}\log(K)^{1-\frac{1}{k}} \qquad (j<D),
    \]
    under the observation that the upper bound in Corollary \ref{L2-T_i,H_i-bounds-lemma} simplifies to $K\hi^{1-k}$ when we have $\hi \ll K^{\frac{1}{k+1}}$. Applying this to \eqref{crit-upper-decomp-eq-2} yields the result.
\end{proof} 

The lower bound
\[\int_0^1 |S_k(\alpha;K)|^{1+\frac{1}{k}} \d\alpha \gg K^{\frac{1}{k}}\]
follows from Lemma \ref{Ls-lower-subcrit-lemma}, but we wish to improve this by a logarithmic factor. A na\"ive adoption of the methods from the case $K=N$ given in \cite{keil:2013} is stopped again by the barrier $K\gg N^{\frac{1}{2}+\epsilon}$, and so we must be a bit more careful in order to prove the result under the weaker restriction $K \gg N^{(\frac{k+1}{2})\Delta_k+\epsilon}$. 

\begin{lemma}\label{crit-h_D-pw-lemma}
    For $c<\frac{1}{2(k+1)}$, let $0\leq a < q^k \leq K^c$ be integers such that $(a,q^k) = 1$ and $q$ is squarefree, and let $|\alpha-\frac{a}{q^k}|< \frac{1}{100K}$. Then one has that $|h_D(\alpha)| \gg \frac{K}{q^k}$ for $N$ sufficiently large, where the implicit constant depends only on $k$.
\end{lemma}
\begin{proof}
    Write $\alpha = \frac{a}{q^k}+\beta$ for $|\beta|<\frac{1}{100K}$. Now by definition, we have that
    \[h_D(\alpha) = \sum_{d\leq K^{1/(k+1)}} \mu(d) \sum_{\frac{N-K}{d^k} < n \leq \frac{N}{d^k}} e\bigg(\frac{nd^ka}{q^k}+nd^k\beta\bigg).\]
    We may then sort the inner sum according to congruence modulo $q^k$, which gives us that
    \[h_D(\alpha) = \sum_{d\leq K^{1/(k+1)}} \mu(d) \bigg(\sum_{0 \leq \ell<q^k} e(\ell d^k\alpha)\sum_{0 \leq m \leq \frac{K}{d^kq^k}} e(mq^kd^k\beta) + O(q^k)\bigg).\]
    Examining the sum over $\ell$, we see that
    \[\bigg|\sum_{0\leq \ell<q^k} e(\ell d^k\alpha) - \sum_{0\leq \ell<q^k} e\bigg(\frac{\ell d^k a}{q^k}\bigg)\bigg| = \bigg|\sum_{0 \leq \ell<q^k} e\bigg(\frac{\ell d^k a}{q^k}\bigg)(e(\ell d^k\beta)-1)\bigg| \ll |q^k d^k\beta|.\]
    Thus our expression for $h_D(\alpha)$ may be written as
    \[h_D(\alpha) = \sum_{d\leq K^{1/(k+1)}} \mu(d) \bigg(\sum_{0 \leq \ell<q^k}e\bigg(\frac{\ell d^ka}{q^k}\bigg) + O(q^kd^k|\beta|)\bigg)\sum_{0\leq m \leq \frac{K}{d^kq^k}} e(mq^kd^k\beta) + O(q^kK^{\frac{1}{k+1}}).\]
    
    Similarly, since we have set $|\beta|<\frac{1}{100K}$, we have that the argument of the innermost exponential is bounded in absolute value by $\frac{1}{100}$. We also have that $d^kq^k < K$, and so
    \[\sum_{0 \leq m \leq \frac{K}{d^kq^k}} e(mq^kd^k\beta) \asymp \frac{K}{d^kq^k},\]
    which simplifies the above expression to
    \[h_D(\alpha) = \dfrac{K}{q^k}\sum_{\substack{d \leq K^{1/(k+1)} \\ q|d}} \dfrac{\mu(d)}{d^k} + O(q^kK^{\frac{1}{k+1}}) + O(K^{1+\frac{1}{k+1}}|\beta|).\]
    The second error term is always dominated by the first, since $|\beta|\ll K^{-1}$. Since $q$ is squarefree, we may extract a factor of $q$ in the main term to write
    \[h_D(\alpha) = \frac{K\mu(q)}{q^{2k}} \sum_{\substack{d \leq N^{1/k}/q \\ (d,q) = 1}} \dfrac{\mu(d)}{d^k} + O(q^kK^{\frac{1}{k+1}}) \]
    with $\mu(q) \neq 0$.
    
    To conclude, we observe that one has that
    \[\sum_{\substack{d \leq N^{1/k}/q \\ (d,q) = 1}} \dfrac{\mu(d)}{d^k} \geq 1-\sum_{\substack{1<d \leq N^{1/k}/q \\ (d,q) = 1}}  \dfrac{1}{d^k} \geq 1-(L(k,\chi_0)-1) \gg 1\]
    where $\chi_0$ is the trivial Dirichlet character modulo $q$, and so one has the desired result whenever $q^{3k}K^{\frac{1}{k+1}} = o(K)$. This is guaranteed since we have chosen $q^k \ll K^c$ with $c< \frac{1}{2(k+1)}$.
\end{proof}

\begin{lemma}\label{crit-lwr-lemma}
    Let $\epsilon>0$. For $K \gg N^{(\frac{k+1}{2})\Delta_k+\epsilon}$, one has that
    \[\int_0^1 |S_k(\alpha;K)|^{1+\frac{1}{k}}\d\alpha \gg K^{\frac{1}{k}}\log K.\]
\end{lemma}
\begin{proof}
    We first show that the integrand can be replaced with $|h_D(\alpha)|^{1+\frac{1}{k}}$ at little cost. To see this, observe that 
    \[|S_k(\alpha;K)|^{1+\frac{1}{k}} - |h_D(\alpha)|^{1+\frac{1}{k}} \leq (|h_D(\alpha)|+|H_D(\alpha)|)^{1+\frac{1}{k}} - |h_D(\alpha)|^{1+\frac{1}{k}}.\]
    The mean value theorem then implies that there exists $0 < \Xi(\alpha) < H_D(\alpha)$ such that
    \[(|h_D(\alpha)|+|H_D(\alpha)|)^{1+\frac{1}{k}} - |h_D(\alpha)|^{1+\frac{1}{k}} \ll |H_D(\alpha)|(|h_D(\alpha)|+\Xi(\alpha))^{\frac{1}{k}},\]
    and since the function $x^{\frac{1}{k}}$ is increasing, we may replace $\Xi(\alpha)$ with $|H_D(\alpha)|$ and see that
    \<\label{S_k-as-h_D-eq}
    \int_0^1|S_k(\alpha;K)|^{1+\frac{1}{k}}\d\alpha = \int_0^1|h_D(\alpha)|^{1+\frac{1}{k}}\d\alpha + O\Bigg( \int_0^1|H_D(\alpha)||h_D(\alpha)|^{\frac{1}{k}}\d\alpha + \int_0^1|H_D(\alpha)|^{1+\frac{1}{k}}\d\alpha\Bigg).
    \>

    We bound these two error terms using H\"older's inequality and the bounds from Corollaries \ref{h_i-bound-lemma} and \ref{L2-T_i,H_i-bounds-lemma}. It is easily seen that
    \[\int_0^1 |H_D(\alpha)|^{1+\frac{1}{k}} \d\alpha \ll \bigg(\int_0^1 |H_D(\alpha)|^2\bigg)^{\frac{k+1}{2k}} \ll (K\hD^{1-k} + N^{\epsilon'} K^{\frac{1}{k}} + N^{\Delta_k+\epsilon'})^{\frac{k+1}{2k}} \ll K^{\frac{1}{k}},\]
    where $\epsilon' = (\frac{2k}{k+1})\epsilon$. For the other error term, we again apply H\"older's inequality to show
    \begin{align*}
        \int_0^1|H_D(\alpha)||h_D(\alpha)|^{\frac{1}{k}}\d\alpha &\ll \bigg(\int_0^1 |H_D(\alpha)|^{\frac{k}{k-1}} \d\alpha\bigg)^{\frac{k-1}{k}}\bigg(\int_0^1 |h_D(\alpha)|\d\alpha\bigg)^{\frac{1}{k}} \\
        &\ll \bigg(\int_0^1 |H_D(\alpha)|^2 \d\alpha\bigg)^{\frac{1}{2}}\bigg(\int_0^1 |h_D(\alpha)|\d\alpha\bigg)^{\frac{1}{k}},
    \end{align*}
    and then again the bounds of Corollaries \ref{h_i-bound-lemma} and \ref{L2-T_i,H_i-bounds-lemma} reveal that
    \begin{align*}
        \int_0^1|H_D(\alpha)||h_D(\alpha)|^{\frac{1}{k}}\d\alpha &\ll (K\hD^{1-k} + N^{\epsilon''}K^{\frac{1}{k}} + N^{\Delta_k+\epsilon''})^{\frac{1}{2}}(\hD\log K)^{\frac{1}{k}} \\
        &\ll K^{\frac{1}{k}}(\log K)^{\frac{1}{k}} + N^{\frac{\Delta_k}{2} + \frac{\epsilon''}{2}}K^{\frac{1}{k(k+1)}}(\log K)^{\frac{1}{k}}.
    \end{align*}
    Again being careful to choose $\epsilon'' = (\frac{2}{k+1})\epsilon$, we see that this is $O((K\log K)^{\frac{1}{k}})$. Thus we have from \eqref{S_k-as-h_D-eq} that 
    \[\int_0^1|S_k(\alpha;K)|^{1+\frac{1}{k}}\d\alpha = \int_0^1|h_D(\alpha)|^{1+\frac{1}{k}}\d\alpha + O((K\log K)^{\frac{1}{k}}),\]
    and it is sufficient then to show that $\int_0^1 |h_D(\alpha)|^{1+\frac{1}{k}} \d\alpha \gg K^{\frac{1}{k}}\log K$. 
    
    For $0\leq a < q^k \leq K^{\frac{1}{4(k+1)}}$ with $(a,q^k)=1$ and $q$ squarefree, define the intervals
    \[\grM(q,a) = \bigg(\frac{a}{q^k}-\frac{1}{100K},\frac{a}{q^k} + \frac{1}{100K}\bigg).\]
    Observe that the intervals are pairwise disjoint, since if $\beta \in \grM(q,a) \cap \grM(r,b)$, then
    \[K^{-\frac{1}{2(k+1)}} \leq \bigg|\frac{ar^k-bq^k}{q^kr^k}\bigg| = \bigg|\frac{a}{q^k} - \frac{b}{r^k}\bigg| \leq \bigg|\beta - \frac{a}{q^k}\bigg| + \bigg|\beta - \frac{b}{r^k}\bigg| < \frac{1}{50}K^{-1},\]
    which is a contradiction for $K$ large enough. Thus we have that
    \[\int_0^1 |h_D(\alpha)|^{1+\frac{1}{k}} \d\alpha \geq \sum_{1\leq q \leq K^{1/(4k+4)}} \mu^2(q) \sum_{\substack{0\leq a < q^k \\ (a,q^k) = 1}}\int_{\grM(q,a)} |h_D(\alpha)|^{1+\frac{1}{k}} \d\alpha,\]
    and by Lemma \ref{crit-h_D-pw-lemma} along with the fact that $|\grM(q,a)| = \frac{1}{50K}$, this becomes
    \<\label{crit-lwr-h_D-sum-eq}
    \int_0^1 |h_D(\alpha)|^{1+\frac{1}{k}} \d\alpha \geq \frac{1}{50}K^{\frac{1}{k}}\sum_{1\leq q^k \leq K^{1/(4k+4)}} \frac{\mu^2(q)\varphi(q^k)}{q^{k+1}} \gg K^{\frac{1}{k}} \sum_{1 \leq q^k \leq K^{1/(4k+4)}} \frac{\mu^2(q)\varphi(q)}{q^2},
    \>
    where $\varphi(q)$ denotes the totient function.

    The conclusion follows from a straightforward computation which is given in \cite{keil:2013}. We may see trivially that
    \[\sum_{1 \leq q \leq Q} \frac{\mu^2(q)\varphi(q)}{q} = \sum_{1 \leq q \leq Q} \frac{\varphi(q)}{q} - \sum_{1 \leq q \leq Q} (1-\mu^2(q))\frac{\varphi(q)}{q},\]
    and applying the identity $\frac{\varphi(q)}{q} = \sum_{d|q} \frac{\mu(d)}{d}$ along with \eqref{mu_k-conv-id-eq} and the bound $\phi(q) \leq q$, we see that this is at least
    \[\frac{Q}{\zeta(2)} + o(Q) - \sum_{1 \leq q \leq Q} (1-\mu^2(q)) = \bigg(\frac{2}{\zeta(2)}-1\bigg)Q + o(Q).\]
    Thus upon applying summation by parts to the sum in \eqref{crit-lwr-h_D-sum-eq}, we may see that
    \[\int_0^1 |h_D(\alpha)|^{1+\frac{1}{k}} \d\alpha \gg K^{\frac{1}{k}}\log K,\]
    as desired.
\end{proof}

Lemmata \ref{crit-upr-lemma} and \ref{crit-lwr-lemma} together give Theorem \ref{crit-thm}. We move on to consideration of the higher moments $s>1+\frac{1}{k}$. For convenience, when $s>1+\frac{1}{k}$ we define
\[
\gamma_k(s) = \begin{cases}
    \big(\frac{1}{2}+\frac{1}{2(s-1)}\big)\Delta_k\quad &\text{if } 1+\frac{1}{k}<s < 2, \\
    \Delta_k \quad &\text{if } s\geq 2.
\end{cases}
\]

\begin{lemma}
    Let $\epsilon>0$ and $s>1+\frac{1}{k}$, and suppose $K \gg N^{\gamma_k(s) + \epsilon}$. Then one has that
    \[\int_0^1 |S_k(\alpha;K)|^s \d\alpha \asymp K^{s-1}.\]
\end{lemma}
\begin{proof}
    It is sufficient via the trivial bound $|S_k(\alpha;K)| \leq K$ to consider the upper bound only when $s \leq 2$. Recall the decomposition
    \<\label{T_i-decomp-eq-supercrit-pf}
    \int_0^1 |S_k(\alpha;K)|^s \d\alpha \ll \int_0^1 |h_D(\alpha)|^s\d\alpha + \int_0^1 |H_D(\alpha)|^s\d\alpha,
    \>
    where $D = (k+1)^{-1}\log(K)$ as in the proof of Lemma \ref{Ls-upper-subcrit-lemma}. For the second integral in \eqref{T_i-decomp-eq-supercrit-pf}, the same argument may be applied as before; since $s\leq 2$ we may apply H\"older's inequality, Parseval's identity, and Lemma \ref{c_n-bd-lemma} to see that
    \<\label{supercrit-lemma-H_D-bd-eq}
    \int_0^1 |H_D(\alpha)|^s\d\alpha \ll \bigg(\int_0^1 |H_D(\alpha)|^2 \d\alpha\bigg)^{\frac{s}{2}} \ll K^{\frac{s}{k+1}} + N^{\frac{s\Delta_k}{2}+\epsilon}.
    \>
    The first term is $O(K^{s-1})$ whenever $s \geq 1+\frac{1}{k}$.

    Now we turn to consideration of the remaining integral on the right of \eqref{T_i-decomp-eq-supercrit-pf}. We have that
    \<\label{supercrit-lemma-low-T_i-decomp-eq}
    \int_0^1 \bigg|\sum_{i<D}T_i(\alpha)\bigg|^s \d\alpha = \int_0^1 \bigg|\sum_{i < D} i^{-1}iT_i(\alpha)\bigg|^s \d\alpha \ll \bigg(\sum_{i < D} i^{-\frac{s}{s-1}}\bigg)^{s-1} \sum_{i<D} i^s\int_0^1|T_i(\alpha)|^s\d\alpha
    \>
    by H\"older's inequality. Since $s\leq 2$, the first sum is $O(1)$. Rather than use directly the $L^s$-bound obtained in Lemma \ref{T_i-bounds-lemma}, we interpolate between the case $s=1+\delta$ of Lemma \ref{T_i-bounds-lemma} and the $L^2$-estimate of Corollary \ref{L2-T_i,H_i-bounds-lemma}. Thus we have that
    \[\int_0^1 |T_i(\alpha)|^s \d\alpha \ll \bigg(\int_0^1 |T_i(\alpha)|^{1+\delta}\d\alpha\bigg)^{\frac{2-s}{1-\delta}}\bigg(\int_0^1 |T_i(\alpha)|^2\d\alpha\bigg)^{\frac{s-1-\delta}{1-\delta}}.\]
    
    Upon applying the bounds from Lemma \ref{T_i-bounds-lemma} and Corollary \ref{L2-T_i,H_i-bounds-lemma}, we then see that this is bounded by
    \[\hi^{\frac{2-s}{1-\delta}}K^{\frac{\delta(2-s)}{1-\delta}} (K\hi^{1-k}+K^{\frac{1}{k}+\epsilon'} + N^{\Delta_k+\epsilon'})^{\frac{s-1-\delta}{1-\delta}}.\]
    Using the standard bound $|x+y|^a \ll |x|^a + |y|^a$, we expand this to
    \[\hi^{\frac{2-s}{1-\delta}}K^{\frac{\delta(2-s)}{1-\delta}} (K\hi^{1-k})^{\frac{s-1-\delta}{1-\delta}}+\hi^{\frac{2-s}{1-\delta}}K^{\frac{\delta(2-s)}{1-\delta}}(K^{\frac{1}{k}+\epsilon'})^{\frac{s-1-\delta}{1-\delta}} + \hi^{\frac{2-s}{1-\delta}}K^{\frac{\delta(2-s)}{1-\delta}}(N^{\Delta_k+\epsilon'})^{\frac{s-1-\delta}{1-\delta}}.\]
    the contribution from the first term simplifies to $K^{s-1}\hi^{1-\frac{k(s-1-\delta)}{1-\delta}}$. Since we may take $\delta$ to be arbitrarily small and $s>1+\frac{1}{k}$, we may write this as $K^{s-1}\hi^{-\phi}$ for some $\phi>0$. Similarly, since we may also take $\epsilon'>0$ arbitrarily small, the second term simplifies to
    \[K^{\frac{s-1 + \delta(k(2-s)-1)}{k(1-\delta)}}\hi^{\frac{2-s}{1-\delta}} \ll K^{\frac{s-1}{k}+\epsilon}\hi^{2-s},\]
    again using the fact that we may take $\delta$ arbitrarily small.

    Finally, the third term contributes
    \[N^{\frac{\Delta_k(s-1-\delta)}{1-\delta}+\epsilon''}K^{\frac{\delta(2-s)}{1-\delta}}\hi^{\frac{2-s}{1-\delta}} \ll N^{\Delta_k(s-1)+\epsilon}\hi^{2-s}.\]
    Applying these bounds to \eqref{supercrit-lemma-low-T_i-decomp-eq}, we have then that
    \begin{align*}
        \int_0^1 \bigg|\sum_{i<D} T_i(\alpha)\bigg|^s \d\alpha &\ll \sum_{i < D} i^s(K^{s-1}\hi^{-\phi} + K^{\frac{s-1}{k}+\epsilon}\hi^{2-s} + N^{\Delta_k(s-1)+\epsilon}\hi^{2-s}) \\
        &\ll K^{s-1} + K^{\frac{s-1+k}{k(k+1)}+\epsilon} + N^{\Delta_k(s-1)+\epsilon}K^{\frac{2-s}{k+1}}.
    \end{align*}
    Observe that since $1+\frac{1}{k}<s\leq 2$, the middle term is at most $K^{\frac{1}{k}+\epsilon} \ll K^{s-1}$. Therefore upon combining this bound with \eqref{T_i-decomp-eq-supercrit-pf} and \eqref{supercrit-lemma-H_D-bd-eq}, we can see that
    \[\int_0^1 |S_k(\alpha;K)|^s \d\alpha \ll K^{s-1}+N^{\frac{s\Delta_k}{2}+\epsilon'}+N^{\Delta_k(s-1)+\epsilon'}K^{\frac{2-s}{k+1}},\]
    and one may easily check that, upon choosing $\epsilon'>0$ small enough, the latter two error terms are dominated by $K^{s-1}$ whenever $K \gg N^{\frac{\Delta_k}{2}+\frac{\Delta_k}{2(s-1)}+\epsilon}$.

    Thus we have the upper bound for all $s>1+\frac{1}{k}$. For the lower bound, we focus first on the case $s=2$. For this case, we imitate the argument made in the proof of Lemma \ref{crit-lwr-lemma}. Observe that, upon expanding $S_k(\alpha;K) = h_D(\alpha) +H_D(\alpha)$, we have that
    \<\label{s=2-lwr-eq-1}
    \int_0^1 |S_k(\alpha;K)|^2 \d\alpha - \int_0^1 |h_D(\alpha)|^2 \d\alpha \ll \int_0^1 |h_D(\alpha)H_D(\alpha)| \d\alpha + \int_0^1 |H_D(\alpha)|^2 \d\alpha.
    \>
    The first integral on the right hand side is $o(K)$ when $K \gg N^{\Delta_k+\epsilon}$ by Lemma \ref{hH-mixed-moment-lemma}. Meanwhile, by Corollary \ref{L2-T_i,H_i-bounds-lemma}, we have that
    \[
    \int_0^1 |H_D(\alpha)|^2 \d\alpha \ll K^{\frac{2}{k+1}} + N^{\Delta_k+\epsilon'},
    \]
    and choosing $\epsilon'<\epsilon$ this is $o(K)$ whenever $K \gg N^{\Delta_k+\epsilon}$. Finally, applying Lemma \ref{crit-h_D-pw-lemma} simply for $q=1$ demonstrates that $\int_0^1 |h_D(\alpha)|^2 \d\alpha \gg K$, and so we have the result for $s=2$.
    
    The lower bound for $s \neq 2$ follows easily by convexity, though we must be slightly careful about the size constraint on $K$. Observe that $\lim_{s \rightarrow 2} \gamma_k(s) = \Delta_k$, and as such for each fixed $\epsilon>0$ we may choose $\delta>0$ sufficiently small that
    \[\int_0^1 |S_k(\alpha;K)|^{2-\delta} \d\alpha \ll K^{1-\delta}\]
    whenever $K \gg N^{\Delta_k+\epsilon}$. For $s>2$, we have that
    \[K \ll \int_0^1 |S_k(\alpha;K)|^2 \d\alpha \leq \bigg(\int_0^1 |S_k(\alpha;K)|^{2-\delta} \d\alpha\bigg)^{\frac{s-2}{s+\delta-2}}\bigg(\int_0^1 |S_k(\alpha;K)|^s\d\alpha\bigg)^{\frac{\delta}{s+\delta-2}},\]
    and so choosing $\delta>0$ sufficiently small with respect to $\epsilon$ we may deduce the lower bound for the $s$-th moment whenever $K \gg N^{\Delta_k+\epsilon}$. A similar argument interpolating the $s$-th moment and the $(2+\delta)$-th moment yields the lower bound for $s<2$.
\end{proof}

\section{The Middle Part Estimate}

Upper bounds on the sum
\[
W(y,z) := \sum_{N-K< n \leq N} |c_n(y,z)|^2 
\]
have been an essential part of $k$-free moment estimates since Balog and Ruzsa \cite{balog-ruzsa:2001}, who showed the bound
\<\label{balog-ruzsa-W-bd}
W(y,z) \ll Ky^{1-k} + N^{\frac{1}{k}}\log^3(z).
\>
The standard approach to estimating $W(y,z)$ is to expand and change the order of summation, applying the triangle inequality to see that
\<\label{W-decomp-1}
    W(y,z) \leq \sum_{N-K < n \leq N} \sum_{\substack{y < d_1,d_2 \leq z \\ d_1^k,d_2^k|n}} 1 = \sum_{y < d_1,d_2 \leq z} \sum_{\substack{N-K < n \leq N \\ [d_1,d_2]^k |n}} 1,
\>
where $[d_1,d_2]$ denotes the least common multiple of $d_1$ and $d_2$. One may then obtain \eqref{balog-ruzsa-W-bd} through estimating the inner sum as $\frac{K}{[d_1,d_2]^k} + 1$.

By the arguments in \S 3, the bound \eqref{balog-ruzsa-W-bd} shows that $\theta_{k,1} \leq \frac{k+1}{2k}$. In the case $k=2$, Sun \cite{sun:2023} further shows that a bound of the form
\<\label{sun-W-bd}
    W(y,z) \ll Ky^{-1} + N^\epsilon K^{\frac{1}{2}} + N^{\frac{12}{29}+\epsilon}y^{-\frac{10}{29}}
\>
may be achieved (though he assumes in his statement that $y \leq K^{\frac{1}{2}-\epsilon}$ so that the middle term may be ignored). The key insight of this improvement is that the trivial estimate on the inner sum of \eqref{W-decomp-1} is weak both when the greatest common divisor $(d_1,d_2)$ is small and when $d_1$ and $d_2$ are large. He applies many of the tools in our repertoire, specifically Perron's formula and a combination of the hyperbola method and van der Corput's method, to trim the error term. 

However, this approach falls short when $y$ is small. Indeed, the optimal bound attainable using only these tools is the bound
\[W(y,z) \ll Ky^{1-k} + N^\epsilon K^{\frac{1}{k}} + N^{\frac{1}{k+1}+\epsilon},\]
at least when one considers the application to the $L^1$-mean. Thus Theorem \ref{Delta-theorem} indicates that in order to show $\theta_{k,1}<\frac{1}{2}$, one must introduce new estimates.

Observe that we may sort the outer sum of \eqref{W-decomp-1} according to $(d_1,d_2)$ and reindex to see that
\[W(y,z) \leq \sum_{1 \leq h \leq z} \sum_{y/h < d_1,d_2 \leq z/h} \sum_{\substack{N-K < n \leq N \\ (d_1d_2h)^k |n}} 1.\]
With this in mind, we define
\<\label{A(H)-def-eq}
A(H) = \sum_{\hH \leq h < 2\hH} \sum_{y/h \leq d_1,d_2 < z/h} \sum_{\substack{N-K < n \leq N \\ (d_1d_2h)^k |n}} 1.
\>
To prove Lemma \ref{c_n-bd-lemma}, we will require five different bounds relating to $A(H)$. First, we observe the trivial bound
\<\label{A(H)-trivial-bd-eq}
A(H) \ll Ky^{2-2k}\hH^{k-1} + z^2\hH^{-1},
\>
which follows from the definition \eqref{A(H)-def-eq} by the bound
\[
\sum_{\substack{N-K < n \leq N \\ (d_1d_2h)^k |n}} 1 \leq \dfrac{K}{(d_1d_2h)^k} + 1.
\]

The next bound takes advantage of the multiplicative structure of the sum when $h$ is small. We essentially follow the argument found in the proof of Lemma 3.1 in \cite{sun:2023}, though we prove it for general $k$.

\begin{lemma}\label{analytic-A(H)-lemma}
    Let $\epsilon>0$. One has that
    \<\label{analytic-A(H)-bd-eq}
    A(H) \ll Ky^{2-2k}\hH^{k-1} + N^{\frac{1}{2}+\epsilon}y^{1-k}\hH^{\frac{k-1}{2}} + N^\epsilon K^{\frac{1}{2}}y^{2-k}\hH^{\frac{k}{2}-1} + N^\epsilon.
    \>
\end{lemma}
\begin{proof}
    It is sufficient to prove the bound when $z=2y$, as each term has a non-positive exponent in $y$. Without loss of generality, we may assume that $N$ and $N-K$ are not integers. Thus we may apply Perron's formula to see that
    \[A(H) = \dfrac{1}{2\pi i }\int_{1+\epsilon -iT_0}^{1+\epsilon+iT_0} \dfrac{N^s-(N-K)^s}{s}\zeta(s)P(ks) \d s + O(N^\epsilon),\]
    where
    \[P(s) = \bigg(\sum_{y/2\hH < d \leq 2y/\hH} d^{-s}\bigg)^2\bigg(\sum_{\hH < h \leq 2\hH} h^{-s}\bigg)\]
    and $T_0 \asymp N$. We proceed by shifting the contour to $\Re(s) = \frac{1}{2}$.

    To handle the horizontal contours, we observe that, by Lemma \ref{pw-zeta-bd-lemma},
    \[\int_{1/2 + iT_0}^{1+\epsilon+iT_0} \dfrac{N^s-(N-K)^s}{s}\zeta(s)P(ks) \d s \ll \sup_{1/2 \leq \sigma \leq 1+\epsilon}\dfrac{N^\sigma}{T_0}T_0^{\frac{1}{3}(1-\sigma)+\epsilon}\bigg(\dfrac{y^2}{\hH}\bigg)^{1-k\sigma} \ll N^\epsilon.\]
    The other horizontal contour may similarly be shown to be $O(N^\epsilon)$. The residue of the integrand at $s=1$ is $O(Ky^{2-2k}\hH^{k-1})$, and so we have that
    \<\label{analytic-prf-int-bd}
    A(H) \ll \int_{1/2+\epsilon -iT_0}^{1/2+\epsilon+iT_0} \bigg|\dfrac{N^s-(N-K)^s}{s}\zeta(s)P(ks)\bigg| \d s + Ky^{2-2k}\hH^{k-1}+N^\epsilon.
    \>

    It therefore remains to bound the integral, which we denote by $J$. Making a change of variable, we see that
    \[J \ll \int_{-T_0}^{T_0} \bigg|\dfrac{N^{1/2+it}-(N-K)^{1/2+it}}{1/2+it}\zeta\bigg(\frac{1}{2}+it\bigg)P\bigg(\frac{k}{2} + kit\bigg)\bigg|\d t.\]
    Noting that the denominator $\frac{1}{2}+it$ is bounded away from zero, we may further bound this by
    \begin{align*}
        J &\ll \int_{-T_0}^{T_0} \min\bigg\{\dfrac{K}{N^{1/2}}, \dfrac{N^{1/2}}{|t|}\bigg\}\bigg|\zeta\bigg(\frac{1}{2}+it\bigg)P\bigg(\frac{k}{2} + kit\bigg)\bigg|\d t \\
        &\ll N^{\frac{1}{2}}\Bigg(\dfrac{K}{N}\int_{-N/K}^{N/K} \bigg|\zeta\bigg(\frac{1}{2}+it\bigg)P\bigg(\frac{k}{2} + kit\bigg)\bigg|\d t +\int_{N/K}^{T_0} |t|^{-1}\bigg|\zeta\bigg(\frac{1}{2}+it\bigg)P\bigg(\frac{k}{2} + kit\bigg)\bigg|\d t\Bigg).
    \end{align*}

    We perform a dyadic decomposition of the second integral to bound it by
    \[\log(N)\sup_{N/K < T < T_0/2} \bigg(T^{-1}\int_{T}^{2T} \bigg|\zeta\bigg(\frac{1}{2}+it\bigg)P\bigg(\frac{k}{2} + kit\bigg)\bigg|\d t\bigg),\]
    and observing that the integrand is non-negative and bounded we may extend the region of integration to $[-2T,2T]$ so that the integral over $[-\frac{N}{K},\frac{N}{K}]$ is dominated by this term. Thus
    \[J \ll N^{\frac{1}{2}}\log(N) \sup_{N/K < T < T_0/2} \bigg(T^{-1}\int_{-2T}^{2T} \bigg|\zeta\bigg(\frac{1}{2}+it\bigg)P\bigg(\frac{k}{2} + kit\bigg)\bigg|\d t\bigg).\]

    We apply the Cauchy-Schwarz inequality to obtain the bound
    \[J \ll N^{\frac{1}{2}}\log(N) \sup_{N/K < T < T_0/2} T^{-1} \bigg(\int_{-2T}^{2T} \bigg|\zeta\bigg(\frac{1}{2}+it\bigg)\bigg|^2 \d t\bigg)^{\frac{1}{2}}\bigg(\int_{-2T}^{2T} \bigg|P\bigg(\frac{k}{2} + kit\bigg)\bigg|^2 \d t\bigg)^{\frac{1}{2}}.\]
    Both integrals may be directly bounded. We have by Lemma \ref{L2-zeta-bd-lemma} that
    \[\int_{-2T}^{2T} \bigg|\zeta\bigg(\frac{1}{2}+it\bigg)\bigg|^2 \d t \ll T\log(T),\]
    and we also have that
    \[\int_{-2T}^{2T} \bigg|P\bigg(\frac{k}{2} + kit\bigg)\bigg|^2 \d t \ll (T+O(y^2\hH^{-1}))\sum_{n>y^2\hH^{-1}} \bigg|\dfrac{r(n)^2}{n^k}\bigg|,\]
    by Lemma \ref{dirichlet-poly-bd-lemma}, where $r(n)$ is the number of representations of $n$ as $d_1d_2h$ for $y/\hH < d_i \leq 2y/\hH$ and $\hH < h \leq 2\hH$. Thus we have that
    \[\int_{-2T}^{2T} \bigg|P\bigg(\frac{k}{2} + kit\bigg)\bigg|^2 \d t \ll N^{\epsilon}Ty^{2(1-k)}\hH^{k-1}+N^{\epsilon}y^{2(2-k)}\hH^{k-2}.\]

    Combining this with the bound for the other integral, we obtain
    \begin{align*}
    J&\ll N^{\frac{1}{2}+\epsilon}\sup_{N/K < T < T_0/2} T^{-1} \bigg(T\log(T)\bigg)^{1/2}\bigg(Ty^{2(1-k)}\hH^{k-1}+y^{2(2-k)}\hH^{k-2}\bigg) \\
    &\ll  N^{\frac{1}{2}+\epsilon}y^{1-k}\hH^{\frac{k-1}{2}} + N^{\epsilon}K^{\frac{1}{2}}y^{2-k}\hH^{\frac{k}{2}-1}.
    \end{align*}
    Returning to \eqref{analytic-prf-int-bd}, we are done.
\end{proof}

\begin{remark}
    The term $N^\epsilon K^{\frac{1}{2}} y^{2-k}\hH^{\frac{k}{2}-1}$ is dominated by either the first or second terms in \eqref{analytic-A(H)-bd-eq} whenever $K \ll N^{\frac{k}{k+1}}$. 
\end{remark}

The remaining bounds are derived using the method of van der Corput. None of these are sufficient on their own, each taking advantage of particular sizes of $h$ and of $d_1,d_2$.

\begin{lemma}\label{vdc1-A(H)-lemma}
    Suppose $(p,q)$ is an exponent pair as in \eqref{exp-pair-def-eq}, and let $\epsilon>0$. If $z = O(y)$, then one has that
    \[A(H) \ll Ky^{2-2k}\hH^{k-1} + N^{-1}y^{2k}\hH^{1-k} +  N^{\frac{p}{p+1}+\epsilon}y^{2-\frac{2kp}{p+1}}\hH^{\frac{1+2q+2kp}{2(p+1)}-2}.\]
\end{lemma}
\begin{proof}
    We may assume again without loss of generality that $z=2y$. The proof follows from a relatively straightforward application of Lemmata \ref{Kusmin-Landau-lemma} and \ref{sv-vdc-bd-lemma}. Observe that we may write
    \begin{align*}
        A(H) &= \sum_{\hH \leq h < 2\hH} \sum_{y/2\hH < d_1,d_2 \leq 2y/\hH} \sum_{\frac{N-K}{(d_1d_2h)^k} < a \leq \frac{N}{(d_1d_2h)^k}} 1 \\
        &= \sum_{\hH \leq h < 2\hH} \sum_{y/2\hH < d_1,d_2 \leq 2y/\hH} \Bigg(\dfrac{K}{(d_1d_2h)^k} + \psi\bigg(\dfrac{N-K}{(d_1d_2h)^k}\bigg) - \psi\bigg(\dfrac{N}{(d_1d_2h)^k}\bigg)\Bigg).
    \end{align*}
    Separating off the first term in the sum, we then find that
    \<\label{sv-vdc-proof-xi}
    A(H) \ll Ky^{2-2k}\hH^{k-1} +  \sum_{y/2\hH < d_1,d_2 \leq 2y/\hH} \Xi(d_1,d_2),
    \>
    where
    \[\Xi(d_1,d_2) = \sum_{\hH \leq h < 2\hH}\psi\bigg(\dfrac{N-K}{(d_1d_2h)^k}\bigg) - \sum_{\hH \leq h < 2\hH}\psi\bigg(\dfrac{N}{(d_1d_2h)^k}\bigg).\]

    If $\hH < y^{\frac{2k}{k-1}}N^{-\frac{1}{k-1}}$, then we may apply Lemma \ref{sawtooth-lemma} with $R=\hH$ and Lemma \ref{Kusmin-Landau-lemma} to each of the sawtooth functions to obtain
    \<\label{Xi(d_1,d_2)-bd-1-eq}
    \Xi(d_1,d_2) \ll N^{-1}y^{2k}\hH^{1-k}.
    \>
    If instead $\hH \geq y^{\frac{2k}{k-1}}N^{-\frac{1}{k-1}}$, then we observe that for the phase functions given by $f_1(x) = (N-K)/(d_1d_2x)^k$ and $f_2(x) = N/(d_1d_2x)^k$, we have
    \[|f_1^{(j)}(x)| \asymp |f_2^{(j)}(x)| \asymp_j \bigg(\dfrac{N}{((y/\hH)^2\hH)^k}\bigg)\hH^{-j} = \bigg(\dfrac{N\hH^k}{y^{2k}}\bigg)\hH^{-j}\]
    whenever $\hH \leq h < 2\hH$. Thus we may apply Lemma \ref{sv-vdc-bd-lemma} to each of these sawtooth sums so that for any exponent pair $(p,q)$, we have
    \<\label{Xi(d_1,d_2)-bd-2-eq}
    \Xi(d_1,d_2) \ll N^{\frac{p}{p+1}}y^{\frac{-2kp}{p+1}}\hH^{\frac{1+2q+2kp}{2(p+1)}+\epsilon}.
    \>
    Applying the two bounds \eqref{Xi(d_1,d_2)-bd-1-eq} and \eqref{Xi(d_1,d_2)-bd-2-eq} to \eqref{sv-vdc-proof-xi}, we have the desired result.
\end{proof}

This bound is sufficient for the majority of our range. However, it fails just barely in the case $k=2$ when $y$ is large, owing to a large contribution from the term $N^{-1}y^{2k}\hH^{1-k}$. Since this term arises when the sawtooth sum is long, we may apply the hyperbola method to shorten the sum. This comes at a cost; the next bound is sufficient for $\hH$ and $y$ large but is rarely a suitable replacement for Lemma \ref{vdc1-A(H)-lemma}. We will only use the following lemma in the case $k=2$ and $(p,q) = (\frac{2}{7},\frac{1}{14})$.

\begin{lemma}\label{vdc1+hyperb-A(H)-lemma}
    Suppose $(p,q)$ is an exponent pair as in \eqref{exp-pair-def-eq} with $1+2q-2kp \geq 0$, and let $\epsilon>0$. If $z = O(y)$, then for any $G>0$, one has that
    \[\sum_{H>G}A(H) \ll Ky^{1-k} + KN^{\frac{1-k}{k+1}} + N^{\frac{1+2q+2p}{2(k+1)(p+1)}}(y/\hG)^{2-\frac{k(1+2q+2p)}{(k+1)(p+1)}}.\]
\end{lemma}
\begin{proof}
    For convenience, let $F_1 = \frac{N-K}{(d_1d_2)^k}$ and $F_2 = \frac{N}{(d_1d_2)^k}$. We begin by writing
    \<\label{hyperb-proof-eq-1}
    \sum_{H>G} A(H) \ll \sum_{1 \leq d_1 < 2y/\hG }\sum_{d_1 \leq d_2 \leq 2d_1} \sum_{\substack{F_1 < h^ka \leq F_2 \\ h \geq y/d_2}} 1 = \sum_{1 \leq d_1 < 2y/\hG}\sum_{d_1 \leq d_2 \leq 2d_1} \Xi(d_1,d_2).
    \>
    Applying Dirichlet's hyperbola method, we may see that
    \<\label{hyperb-proof-eq-2}
    \Xi(d_1,d_2) \ll \sum_{y/d_2 < h \leq F_2^{1/(k+1)}} \sum_{F_1/h^k < a \leq F_2/h^k} 1 + \sum_{a \leq F_2^{1/(k+1)}} \sum_{(F_1/a)^{1/k} <h \leq (F_2/a)^{1/k}} 1.
    \>
    
    Denote these two double sums by $\Xi_1(d_1,d_2)$ and $\Xi_2(d_1,d_2)$, respectively. Now we apply the definition of the sawtooth function to see that
    \<\label{hyperb-xi1-eq-1}
    \Xi_1(d_1,d_2) \ll \sum_{y/d_2 < h \leq F_2^{1/(k+1)}} \frac{K}{(hd_1d_2)^k} + \Bigg|\sum_{y/d_2 < h \leq F_2^{1/(k+1)}} \psi\bigg(\dfrac{F_1}{h^k}\bigg)\Bigg| + \Bigg|\sum_{y/d_2 < h \leq F_2^{1/(k+1)}} \psi\bigg(\dfrac{F_2}{h^k}\bigg)\Bigg|.
    \>
    The sawtooth sums are handled identically by Lemma \ref{sv-vdc-bd-lemma}. We perform a dyadic dissection to show that 
    \[\Bigg|\sum_{y/d_2 < h \leq F_2^{1/(k+1)}} \psi\bigg(\dfrac{F_1}{h^k}\bigg)\Bigg| \ll \log(N) \max_{y/d_2 < M \leq F_2^{1/(k+1)}} \Bigg|\sum_{M < h \leq 2M} \psi\bigg(\dfrac{F_1}{h^k}\bigg)\Bigg|,\]
    and apply Lemma \ref{sv-vdc-bd-lemma} to see that this in turn is bounded by
    \[\log(N)\max_{y/d_2 < M \leq F_2^{1/(k+1)}} F_2^{\frac{p}{p+1}+\epsilon}M^{\frac{1+2q-2kp}{2(p+1)}} \ll F_2^{\frac{1+2q+2p}{2(k+1)(p+1)}+\epsilon},\]
    since we have assumed that $1+2q-2kp \geq 0$ and $F_1 \asymp F_2$. Applying this argument to both sums in \eqref{hyperb-xi1-eq-1}, we obtain the bound
    \[\Xi_1(d_1,d_2) \ll \frac{Ky^{1-k}}{d_1^kd_2} + \bigg(\dfrac{N}{(d_1d_2)^k}\bigg)^{\frac{1+2q+2p}{2(k+1)(p+1)}+\epsilon}.\]

    The sawtooth identity applied to $\Xi_2(d_1,d_2)$, on the other hand, shows that
    \<\label{hyperb-xi2-eq-1}
    \Xi_2(d_1,d_2) \ll \sum_{a \leq F_2^{1/(k+1)}}  \bigg(\dfrac{F_2}{a}\bigg)^{\frac{1}{k}}-\bigg(\dfrac{F_1}{a}\bigg)^{\frac{1}{k}} + \Bigg|\sum_{a \leq F_2^{1/(k+1)}} \psi\bigg(\bigg(\dfrac{F_1}{a}\bigg)^{\frac{1}{k}}\bigg)\Bigg| + \Bigg|\sum_{a \leq F_2^{1/(k+1)}} \psi\bigg(\bigg(\dfrac{F_2}{a}\bigg)^{\frac{1}{k}}\bigg)\Bigg|.
    \>
    Again, we may perform a dyadic decomposition and apply Lemma \ref{sv-vdc-bd-lemma} to each sawtooth sum to see that
    \begin{align*}
        \Bigg|\sum_{a \leq F_2^{1/(k+1)}} \psi\bigg(\bigg(\dfrac{F_1}{a}\bigg)^{\frac{1}{k}}\bigg)\Bigg| &\ll \log(N) \max_{1 \leq M \leq F_2^{1/(k+1)}} \Bigg|\sum_{M < a \leq 2M} \psi\bigg(\bigg(\dfrac{F_1}{a}\bigg)^{\frac{1}{k}}\bigg)\Bigg| \\
        &\ll F_2^{\frac{1+2q+2p}{2(k+1)(p+1)} + \epsilon},
    \end{align*}
    again since we have chosen $(p,q)$ with $k+2kq-p \geq 1+2q-2kp\geq 0$. Returning to \eqref{hyperb-xi2-eq-1}, we may then observe that
    \[\Xi_2(d_1,d_2) \ll \dfrac{KN^{\frac{1-k}{k+1}}}{(d_1d_2)^{\frac{2k}{k+1}}} +\bigg(\dfrac{N}{(d_1d_2)^k}\bigg)^{\frac{1+2q+2p}{2(k+1)(p+1)}+\epsilon}.\]

    Combining this and the bound on $\Xi_1(d_1,d_2)$ with \eqref{hyperb-proof-eq-1} and \eqref{hyperb-proof-eq-2}, we have that
    \begin{align*}
        \sum_{H>G} A(H) &\ll \sum_{1 \leq d_1 < 2y/\hG}\sum_{d_1 \leq d_2 \leq 2d_1} \frac{Ky^{1-k}}{d_1^kd_2} + \dfrac{KN^{\frac{1-k}{k+1}}}{(d_1d_2)^{\frac{2k}{k+1}}} + \bigg(\dfrac{N}{(d_1d_2)^k}\bigg)^{\frac{1+2q+2p}{2(k+1)(p+1)}+\epsilon} \\
        &\ll Ky^{1-k} + KN^{\frac{1-k}{k+1}} + N^{\frac{1+2q+2p}{2(k+1)(p+1)}+\epsilon}(y/\hG)^{2-\frac{k(1+2q+2p)}{(k+1)(p+1)}}
    \end{align*}
    as desired.
\end{proof}

The final bound on $A(H)$ arrives from an application of the two-dimensional van der Corput method, applied in a manner very similar to Lemma \ref{vdc1-A(H)-lemma}. This is the key bound which allows us to eclipse the square-root barrier.

\begin{lemma}\label{vdc2-A(H)-lemma}
    Suppose $\epsilon>0$ and $z=O(y)$. One has that
    \[A(H) \ll Ky^{2-2k}\hH^{k-1} + N^{\frac{1}{4}}y^{\frac{5-2k}{4}}\hH^{\frac{k-1}{4}} + N^\epsilon y^{\frac{5}{3}}\hH^{-\frac{2}{3}} + N^{-\frac{1}{8}+\epsilon}y^{\frac{15+2k}{8}}\hH^{-\frac{7+k}{8}} + N^{-\frac{1}{4}+\epsilon}y^{2+\frac{k}{2}}\hH^{-\frac{k}{4}-1}.\]
\end{lemma}
\begin{proof}
    As in the last argument, we write
    \[A(H) = \sum_{\hH \leq h < 2\hH} \sum_{y/\hH < d_1,d_2 \leq 2y/\hH} \Bigg(\dfrac{K}{(d_1d_2h)^k} + \psi\bigg(\dfrac{N-K}{(d_1d_2h)^k}\bigg) - \psi\bigg(\dfrac{N}{(d_1d_2h)^k}\bigg)\Bigg).\]
    Now separating off the main term again, we instead write
    \<\label{mv-vdc-xi}
    A(H) \ll Ky^{2-2k}\hH^{k-1} +  \sum_{\hH < h \leq 2\hH} \Xi_0(h),
    \>
    where
    \[\Xi_0(h) = \sum_{y/\hH \leq d_1,d_2 < 2y/\hH}\psi\bigg(\dfrac{N-K}{(d_1d_2h)^k}\bigg) - \sum_{y/\hH \leq d_1,d_2 < 2y/\hH}\psi\bigg(\dfrac{N}{(d_1d_2h)^k}\bigg).\]

    We will bound the second sum; the first sum is bounded by exactly the same reasoning. We have by Lemma \ref{sawtooth-lemma} that for any positive integer $R$,
    \[\sum_{y/\hH \leq d_1,d_2 < 2y/\hH}\psi\bigg(\dfrac{N}{(d_1d_2h)^k}\bigg) \ll (y/\hH)^2R^{-1} + \sum_{1 \leq |r| \leq R} r^{-1}\Bigg|\sum_{y/\hH \leq d_1,d_2 < 2y/\hH}e\bigg(\dfrac{rN}{(d_1d_2h)^k}\bigg)\Bigg|.\]
    Denote this exponential sum by $S$. We may then apply Theorem \ref{vdc2-ref-lemma} to see that
    \[
        S \ll r^{\frac{1}{3}}N^{\frac{1}{3}}y^{1-\frac{2k}{3}}\hH^{\frac{k}{3}-1} + N^\epsilon y^{\frac{5}{3}}\hH^{-\frac{5}{3}} + r^{-\frac{1}{8}}N^{\epsilon-\frac{1}{8}}y^{\frac{15+2k}{8}}\hH^{-\frac{15+k}{8}}
        + r^{-\frac{1}{4}}N^{\epsilon-\frac{1}{4}}y^{2+\frac{k}{2}}\hH^{-\frac{k}{4}-2}.
    \]
    Thus upon choosing $R = N^{-\frac{1}{4}}y^{\frac{2k+3}{4}}\hH^{-\frac{k+3}{4}}$, we have that
    \begin{multline*}
        \sum_{y/\hH \leq d_1,d_2 < 2y/\hH}\psi\bigg(\dfrac{N}{(d_1d_2h)^k}\bigg) \ll N^{\frac{1}{4}}y^{\frac{5-2k}{4}}\hH^{\frac{k-5}{4}} + N^\epsilon y^{\frac{5}{3}}\hH^{-\frac{5}{3}}+ N^{\epsilon-\frac{1}{8}}y^{\frac{15+2k}{8}}\hH^{-\frac{15+k}{8}} \\
        + N^{\epsilon-\frac{1}{4}}y^{2+\frac{k}{2}}\hH^{-\frac{k}{4}-2},
    \end{multline*}
    which yields the desired result upon substitution into \eqref{mv-vdc-xi}.
\end{proof}

We are finally prepared to prove the main bound on $W(y,z)$. We prove a more general result in terms of a generic exponent pair $(p,q)$, and the result follows by specifying an exponent pair.

\begin{lemma}\label{c_n-(p,q)-bd-lemma}
    Let $k \geq 2$ and $1 \leq y < z \leq N^{\frac{1}{k}}$, and let $(p,q)$ be an exponent pair as defined in \eqref{exp-pair-def-eq} satisfying $3+(4-2k)p-2q \geq 0$. Then one has that
    \<\label{W(y,z)-claimed-bound-eq}
     W(y,z) \ll Ky^{1-k} + N^\epsilon K^{\frac{1}{k}} +  N^{\frac{3(2p-2q+3)}{2((k+5)p-(4k+2)q+5k+4)}+\epsilon}.
    \>
\end{lemma}
\begin{proof}
    We first show that it is sufficient to prove the result when $z=2y$. One must be careful here, as it is important not to pick up a logarithmic factor on the term $Ky^{1-k}$. This is dependent on the shape of the bound. Suppose that
    \<\label{last-lemma-assumed-dyadic-eq}
    W(\hY,2\hY) \ll K\hY^{1-k} + N^\epsilon K^{\frac{1}{k}} + N^{\Delta_k+\epsilon}
    \>
    for each $\hY$ in our range, and consider the complete sum $W(y,z)$ for some $1 \leq y < z\leq N^{\frac{1}{k}}$. We expand to see that
    \begin{align*} 
    W(y,z) &\leq \sum_{N-K< n \leq N} \bigg|\sum_{\log(y) \leq Y \leq \log(z)}c_n(\hY,2\hY)\bigg|^2  \\
    &\leq \sum_{\log(y)\leq Y_1,Y_2 \leq \log(z)}\sum_{N-K<n \leq N}  c_n(\hY_1,2\hY_1)c_n(\hY_2,2\hY_2).
    \end{align*}
    We then apply the Cauchy-Schwarz inequality and \eqref{last-lemma-assumed-dyadic-eq}, deducing that
    \begin{align*}
        W(y,z) &\ll \bigg(\sum_{\log(y)\leq Y \leq \log(z)} (K\hY^{1-k} + N^\epsilon K^{\frac{1}{k}} + N^{\Delta_k+\epsilon})^{\frac{1}{2}}\bigg)^2 \ll Ky^{1-k} + N^\epsilon K^{\frac{1}{k}} + N^{\Delta_k+\epsilon}.
    \end{align*}
    Thus we may reduce to the case of $W(\hY,2\hY)$.

    Assume now that $K \gg N^{\frac{k}{k+1}}$. We write
    \[W(\hY,2\hY) \ll \sum_{H < L} A(H) + \sum_{H \geq L} A(H),\]
    where $\hL = \hY^2K^{-\frac{1}{k}}$. Applying Lemma \ref{analytic-A(H)-lemma} to the first set of $A(H)$ and the trivial bound \eqref{A(H)-trivial-bd-eq} to $A(H)$ with $H \geq L$, we have that
    \[W(\hY,2\hY) \ll K\hY^{1-k} + N^{\frac{1}{2}+\epsilon}K^{\frac{1-k}{2k}} + N^\epsilon K^{\frac{1}{k}}.\]
    Since $K \gg N^{\frac{k}{k+1}}$, the middle term is $O(N^{\epsilon}K^{\frac{1}{k}})$.

    The case in which $K \ll N^{\frac{k}{k+1}}$ is more computationally intensive, although the underlying principle is the same. We use the various bounds for $A(H)$ to construct bounds for $W(\hY,2\hY)$ which are effective in various ranges. The first bound is obtained by applying Lemmata \ref{analytic-A(H)-lemma}, \ref{vdc1-A(H)-lemma}, and \ref{vdc2-A(H)-lemma}. 

    We write
    \[W(\hY,2\hY) \ll \sum_{H<L_1} A(H) + \sum_{L_1 \leq H < L_2} A(H) + \sum_{H \geq L_2} A(H) =: S_1 +S_2 + S_3.\]
    Lemma \ref{analytic-A(H)-lemma} applied to $S_1$ gives again that
    \[S_1 \ll K\hY^{1-k} + N^{\frac{1}{2}+\epsilon}\hY^{1-k}\hL_1^{\frac{k-1}{2}},\]
    while Lemma \ref{vdc2-A(H)-lemma} applied to $S_2$ yields
    \[S_2 \ll K\hY^{1-k} + N^{\frac{1}{4}}y^{\frac{5-2k}{4}}\hL_2^{\frac{k-1}{4}} + N^\epsilon y^{\frac{5}{3}}\hL_1^{-\frac{2}{3}} + N^{-\frac{1}{8}+\epsilon}y^{\frac{15+2k}{8}}\hL_1^{-\frac{7+k}{8}} + N^{-\frac{1}{4}+\epsilon}y^{2+\frac{k}{2}}\hL_1^{-\frac{k}{4}-1}\]
    and
    \[S_3 \ll K\hY^{1-k} + N^{-1}\hY^{2k}\hL_2^{1-k} +  N^{\frac{p}{p+1}+\epsilon}\hY^{2-\frac{2kp}{p+1}}\hL_2^{\frac{1+2q+2kp}{2(p+1)}-2}\]
    by Lemma \ref{vdc1-A(H)-lemma}. Choosing
    \[\hL_1 = \min\bigg\{N^{-\frac{3}{3k+1}}\hY^{\frac{6k+14}{3k+1}}, N^{-\frac{5}{5k+3}}\hY^{\frac{10k+7}{5k+3}}, N^{-\frac{3}{3k+2}}\hY^2\bigg\}\]
    and
    \[\hL_2 = N^{\frac{3p-1}{(7-3k)p-4q+k+5}}\hY^{\frac{(3-6k)p+2k+3}{(7-3k)p-4q+k+5}},\]
    we may deduce that
    \<\label{W(Y,2Y)-up-bd-final,Y<Y_1}
    W(\hY,2\hY) \ll K\hY^{1-k} + \hY^{\frac{2}{3}+\epsilon} + N^{a_1+\epsilon}\hY^{a_2} + N^{\frac{2}{3k+1}+\epsilon}\hY^{\frac{k-1}{3k+1}} + N^{\frac{4}{5k+3}+\epsilon}\hY^{\frac{k-1}{2(5k+3)}} + N^{\frac{5}{2(3k+2)}+\epsilon},
    \>
    where
    \<\label{a_1-a_2-def-eq}
    a_1 = \dfrac{2p-2q+3}{2((7-3k)p-4q+k+5)}, \quad a_2 = \dfrac{(16-10k)p+(4k-10)q-2k+11}{2((7-3k)p-4q+k+5)}.
    \>

    The bound \eqref{W(Y,2Y)-up-bd-final,Y<Y_1} is most effective when $\hY$ is small. In fact, one finds that the use of Lemma \ref{vdc2-A(H)-lemma} is ineffective when 
    \[\hY \gg N^{\frac{2p-2q+3}{(k+5)p-(4k+2)q+5k+4}},\]
    as this is where $\hL_1 = \hL_2$. Thus we also make use of a bound constructed using only Lemmata \ref{analytic-A(H)-lemma} and \ref{vdc1-A(H)-lemma}. This results in the bound
    \begin{align*}
        W(\hY,2\hY) &\ll \sum_{H < L} A(H) + \sum_{H \geq L} A(H) \\
        &\ll K\hY^{1-k} + N^{\frac{1}{2}+\epsilon}\hY^{1-k}\hL^{\frac{k-1}{2}} + N^{-1}\hY^{2k}\hL^{1-k} + N^{\frac{p}{p+1}+\epsilon}\hY^{2-\frac{2kp}{p+1}}\hL^{\frac{1+2q+2kp}{2(p+1)}-2}.
    \end{align*}
    Since we have assumed $3+(4-2k)p-2q \geq 0$, the exponent of $\hL$ in the last term is negative. Thus upon taking $\hL = \max\{N^{-\frac{1}{k-1}}\hY^{\frac{6k-2}{3k-3}}, N^{\frac{p-1}{(3-k)p-2q+k+2}}\hY^{\frac{2(k+1)(p+1)-4kp}{(3-k)p-2q+k+2}}\}$, we deduce
    \<
    W(\hY,2\hY) \ll K\hY^{1-k} + \hY^{\frac{2}{3}+\epsilon} + N^{\frac{2p-2q+3}{2((3-k)p-2q+k+2)}}\hY^{\frac{(k-1)(2q-2p-1)}{(3-k)p-2q+k+2}}. \label{W(Y,2Y)-up-bd-final,Y_1<Y<Y_2}
    \>
    We observe that the factor of $\hY$ in the last term of \eqref{W(Y,2Y)-up-bd-final,Y_1<Y<Y_2} will have a negative exponent for any exponent pair.

    The bounds \eqref{W(Y,2Y)-up-bd-final,Y<Y_1} and \eqref{W(Y,2Y)-up-bd-final,Y_1<Y<Y_2} are sufficient for all but the case $k=2$, where the term $\hY^{\frac{2}{3}+\epsilon}$ is too large. To surpass this barrier, we instead combine Lemmata \ref{analytic-A(H)-lemma} and \ref{vdc1+hyperb-A(H)-lemma}, which gives a bound only effective for $\hY$ very close to $N^{\frac{1}{2}}$. Assuming $k=2$, we apply Lemma \ref{analytic-A(H)-lemma} and Lemma \ref{vdc1+hyperb-A(H)-lemma} with $(p,q) = (\frac{2}{7},\frac{1}{14})$ to deduce that
    \[
        W(\hY,2\hY) \ll \sum_{H < L} A(H) + \sum_{H \geq L} A(H) \ll K\hY^{-1} + N^{\frac{1}{2}+\epsilon}\hY^{-1}\hL^{\frac{1}{2}} + KN^{-\frac{1}{3}} + N^{\frac{2}{9}}(\hY/\hL)^{\frac{10}{9}}
    \]
    and upon choosing $\hL = N^{-\frac{5}{29}}\hY^{\frac{38}{29}}$ we may deduce that
    \<\label{W(Y,2Y)-up-bd-final,Y>Y_2}
    W(\hY,2\hY) \ll K\hY^{-1} + KN^{-\frac{1}{3}} + N^{\frac{12}{29}+\epsilon}\hY^{-\frac{10}{29}}.
    \>
    Since we have assumed $K \ll N^{\frac{k}{k+1}} = N^{\frac{2}{3}}$, the middle term is $O(N^\epsilon K^{\frac{1}{2}})$, and so we recover the bound \eqref{sun-W-bd} from \cite{sun:2023}.

    The proof is concluded by applying the bounds \eqref{W(Y,2Y)-up-bd-final,Y<Y_1} and \eqref{W(Y,2Y)-up-bd-final,Y_1<Y<Y_2}, along with \eqref{W(Y,2Y)-up-bd-final,Y>Y_2} in the case $k=2$, to appropriate ranges of $\hY$. In particular, for each $k$ we set $\hY_1 = N^{m_1}$, where 
    \[m_1 = \dfrac{2p-2q+3}{(k+5)p-(4k+2)q+5k+4},\]
    and we note that for $a_1,a_2$ as defined in \eqref{a_1-a_2-def-eq}, one has
    \[a_1+a_2m_1 = \frac{3(2p-2q+3)}{2((k+5)p-(4k+2)q+5k+4)},\]
    which is the desired exponent in the bound \eqref{W(y,z)-claimed-bound-eq}.
    For $k>2$, we apply \eqref{W(Y,2Y)-up-bd-final,Y<Y_1} if $\hY < \hY_1$ and \eqref{W(Y,2Y)-up-bd-final,Y_1<Y<Y_2} if $\hY \geq \hY_1$. Noting that each of the terms in the bound \eqref{W(Y,2Y)-up-bd-final,Y<Y_1} besides $K\hY^{1-k}$ have positive exponents on $\hY$, we have that for $\hY<\hY_1$,
    \[
        W(\hY,2\hY) \ll K\hY^{1-k} + \hY_1^{\frac{2}{3}+\epsilon} + N^{a_1+\epsilon}\hY_1^{a_2} + N^{\frac{2}{3k+1}+\epsilon}\hY_1^{\frac{k-1}{3k+1}} + N^{\frac{4}{5k+3}+\epsilon}\hY_1^{\frac{k-1}{2(5k+3)}} + N^{\frac{5}{2(3k+2)}+\epsilon},
    \]
    and upon substituting the definition of $\hY_1$ this becomes
    \begin{multline*}
        W(\hY,2\hY) \ll K\hY^{1-k} + N^{\frac{2}{3}m_1+\epsilon} + N^{a_1+a_2m_1+\epsilon} \\
        + N^{\frac{2}{3k+1}+\frac{k-1}{3k+1}m_1+\epsilon} + N^{\frac{4}{5k+3}+\frac{k-1}{2(5k+3)}m_1 + \epsilon} + N^{\frac{5}{2(3k+2)}+\epsilon}.
    \end{multline*}
    Comparing these terms, it is easy to verify that the term $N^{a_1+a_2m_1+\epsilon}$ dominates the other terms of that shape. Thus the above bound simplifies to
    \<\label{W(Y,2Y)-final-eq-1}
    W(\hY,2\hY) \ll K\hY^{1-k} + N^{a_1+a_2m_1+\epsilon} \qquad (\hY<\hY_1).
    \>

    Similarly, recalling that the third term in \eqref{W(Y,2Y)-up-bd-final,Y_1<Y<Y_2} is negative in the exponent of $\hY$, we may deduce that for $\hY \geq \hY_1$ one has
    \begin{align*}
        W(\hY,2\hY) &\ll K\hY^{1-k} + \hY^{\frac{2}{3}+\epsilon} + N^{\frac{2p-2q+3}{2((3-k)p-2q+k+2)}}\hY_1^{\frac{(k-1)(2q-2p-1)}{(3-k)p-2q+k+2}} \\
        &\ll K\hY^{1-k} + \hY^{\frac{2}{3}+\epsilon} + N^{a_1+a_2m_1+\epsilon}. \stepcounter{equation}\tag{\theequation}\label{W(Y,2Y)-almost-final-eq}
    \end{align*}
    Provided $k>2$, the second term is dominated by the third, and so
    \<\label{W(Y,2Y)-final-eq-2}
    W(\hY,2\hY) \ll K\hY^{1-k} + N^{a_1+a_2m_1+\epsilon} \qquad (\hY \geq \hY_1).
    \>
    Together \eqref{W(Y,2Y)-final-eq-1} and \eqref{W(Y,2Y)-final-eq-2} prove that \eqref{W(y,z)-claimed-bound-eq} holds for $k>2$.

    In the case $k=2$, we set $\hY_2 = N^{\frac{9}{22}}$, and observe that \eqref{W(Y,2Y)-almost-final-eq} becomes
    \<\label{W(Y,2Y)-k=2-eq-1}
    W(\hY,2\hY) \ll K\hY^{1-k} + N^{\frac{3}{11}+\epsilon} + N^{a_1+a_2m_1+\epsilon} \qquad (\hY < \hY_2).
    \>
    Similarly, we may see that \eqref{W(Y,2Y)-up-bd-final,Y>Y_2} becomes
    \<\label{W(Y,2Y)-k=2-eq-2}
    W(\hY,2\hY) \ll K\hY^{-1} + N^\epsilon K^{\frac{1}{2}} + N^{\frac{3}{11}+\epsilon} \qquad (\hY \geq \hY_2).
    \>
    Upon observing that $N^{\frac{3}{11}+\epsilon} \ll N^{a_1+a_2m_1+\epsilon}$ for all exponent pairs $(p,q)$, the bounds \eqref{W(Y,2Y)-k=2-eq-1} and \eqref{W(Y,2Y)-k=2-eq-2} together imply \eqref{W(y,z)-claimed-bound-eq}, and we at last have the desired result for all $k \geq 2$.
\end{proof}

\begin{proof}[Proof of Lemma \ref{c_n-bd-lemma}]
    For $k=2$ we apply Lemma \ref{c_n-(p,q)-bd-lemma} with the exponent pair $(\frac{1}{30},\frac{11}{30})$. For $k>2$, the specific values of $\delta_k$ are obtained by using the exponent pairs
    \[\bigg(\frac{2}{k^2(k+3)}, \frac{1}{2}-\frac{3k+1}{k(k+1)(k+3)}\bigg),\]
    which are newly available due to the resolution of the main conjecture of Vinogradov's mean value theorem (see \cite{wooley:2016} and \cite{bourgain-demeter-guth:2016}). These exponent pairs are derived by Heath-Brown in \cite{heath-brown:2017}, though one must be careful to observe that our definition of exponent pairs differs from that used in \cite{heath-brown:2017}.
\end{proof}

We conclude with a remark that almost certainly one may improve this result slightly with more carefully chosen exponent pairs. This is the case for example when $k=3$, where we have shown that $\theta_{3,1} \leq \frac{381}{769} \approx 0.49545$, but the exponent pair $(\frac{1}{14},\frac{2}{7})$ applied to Lemma \ref{c_n-(p,q)-bd-lemma} results in a slightly better bound of $\frac{54}{109} \approx 0.49542$. However, these improvements are not significant (indeed, even with the exponent pairs conjecture, Lemma \ref{c_n-(p,q)-bd-lemma} cannot show $\theta_{k,1} \leq \frac{9}{20}$ for any $k$), and so we opt for a unified approach.

\section*{Acknowledgements}

The author is grateful to his advisor Trevor Wooley for his constant support and suggestions. He would also like to thank Anurag Sahay and Atal Bhargava for helpful discussions. This work was partially supported by NSF grants DMS-2001549 and DMS-2502625, under the supervision of Trevor Wooley.

\bibliography{refs}

\end{document}